\documentclass[12pt]{amsart}
\usepackage{color}
\usepackage{graphicx}
\usepackage{graphicx}
\usepackage{epsfig}
\usepackage{eucal}
\usepackage{tikz}
\usetikzlibrary{arrows}

\newtheorem{thm}{Theorem}[section]

\newtheorem{defn}[thm]{Definition}
\newtheorem{lemma}[thm]{Lemma}

\newtheorem{cor}[thm]{Corollary}
\newtheorem{remark}[thm]{Remark}
\newtheorem{example}[thm]{Example}

\usepackage{amsmath}
\usepackage{amsxtra}
\usepackage{amscd}
\usepackage{amsthm}
\usepackage{amsfonts}
\usepackage{amssymb}
\usepackage{eucal}

\newcommand{\bmb}{\left( \begin{array}{rr}}
\newcommand{\enm}{\end{array}\right)}

\newcommand{\gl}{{\mathfrak{gl}}}
\newcommand{\n}{{\mathfrak{n}}}

\newcommand{\m}{\mathfrak{m}}

\newcommand{\U}{\mathcal U}

\renewcommand{\sl}{{\mathfrak{sl}}}

\newcommand{\C}{{\mathbb C}}

\newcommand{\Z}{{\mathbb Z}}
\newcommand{\Q}{{\mathcal Q}}

\newcommand{\N}{{\mathbb N}}

\newcommand{\bx}{{\mathbf x}}

\newcommand{\bu}{{\mathbf u}}

\newcommand{\al}{{\alpha}}

\numberwithin{equation}{section}

\begin{document}

\title{Quantum Q systems: From cluster algebras to quantum current algebras}
\author{Philippe Di Francesco} 
\address{PDF: Department of Mathematics, University of Illinois MC-382, Urbana, IL 61821, U.S.A. e-mail: philippe@illinois.edu}
\author{Rinat Kedem}
\address{RK: Department of Mathematics, University of Illinois MC-382, Urbana, IL 61821, U.S.A. e-mail: rinat@illinois.edu}
\begin{abstract}
In this paper, we recall our renormalized quantum Q-system associated with representations of the Lie algebra $A_r$, and show that it can be viewed as a quotient of the quantum current algebra $U_q(\n[u,u^{-1}])\subset U_q(\widehat{\sl}_2)$ in the Drinfeld presentation. Moreover, we find the interpretation of the conserved quantities in terms of Cartan currents at level 0, and the rest of the current algebra, in a non-standard polarization in terms of generators in the quantum cluster algebra.
\end{abstract}

\maketitle
\date{\today}
\tableofcontents
\section{Introduction}

The quantum $Q$-system can be regarded in two different ways, as an algebra and as a discrete evolution. In our previous work \cite{DFKnoncom}, we considered it as a discrete dynamical system for a set of non-commuting variables, which forms a quantum integrable system. As such, it has commuting integrals of motion (conserved quantities) and can be solved combinatorially. 

We also considered the quantum $Q$-system as a relation in the non-commutative algebra generated by these dynamical variables \cite{qKR}. 
As such, it is a subalgebra of a quantum cluster algebra, whose generators are non-commuting and invertible. The relations of the cluster algebra are mutations, which can be interpreted as recursion relations, as well as the commutation relations within ``clusters" forming overlapping subsets of the generators.

In \cite{DFK15}, we showed that the generators of the quantum $Q$-system algebra act on the space of symmetric polynomials with coefficients in $\Z[q,q^{-1}]$ ($q$ a central formal variable). The resulting symmetric functions are interpreted as the graded characters \cite{FL} of the Feigin-Loktev fusion product of KR-modules for the polynomial algebra $\sl_n[u]$. The generators of the quantum $Q$-system act on these polynomials as difference operators, mapping the graded character for a given tensor product to that with one extra tensored representation. This gives an inductive rule for computing the characters.

This leaves open the question of the exact nature of the algebra of the quantum $Q$-system, of the exact representation given by the action by difference operators on the graded characters, and of the interpretation of conserved quantities within this algebra.

This paper is the first of a sequence of two. In the second paper \cite{DFK16}, we show
 that there is a natural $t$-deformation of the above difference operators.
These are similar to, but simpler than the raising operators for Macdonald polynomials introduced by Kirillov and Noumi \cite{kinoum}, 
and they are part of a representation of the quantum toroidal algebra of $\gl_1$ at level 0 \cite{FJqtoroidal}, with quantum parameters $q$ and $t$. 

The quantum $Q$-system is the dual $q$-Whittaker limit of this algebra, in the sense of Macdonald theory, namely corresponds to the limit $t\to\infty$. More precisely, when rephrased in terms of current generators, the mutation and commutation relations of the quantum $Q$-system are equivalent to the relations among currents in the nilpotent subalgebra of the quantum enveloping algebra of $\sl_2$. This is the main result of this paper, which we prove directly in the $q$-Whittaker limit. Moreover, we construct in this representation the other generators of the quantized affine algebra of $\sl_2$. 

The representation has certain non-standard features. In particular, the quantum $Q$-system is associated with the rank $r$ of the algebra (type $A_r$). This translates into an extra, rank-dependent relation satisfied by the current generators, which is not present in the quantum affine algebra.

Because of the complexity of the formulas and the proofs, we will present the quantum toroidal deformation in a separate publication. In this paper, we concentrate on the structure of the quantum $Q$-system itself.

The paper is organized as follows. In Section 2, we define the graded algebra $\U_r$ associated with the $A_r$ quantum $Q$-system extended with coefficients and present the main results of the paper, namely that $\U_r$ may be reformulated as a quotient of the nilpotent subalgebra of the quantum affine algebra of $\sl_2$. Section 3 reformulates the conserved quantities of this system and their properties as discrete Hamiltonians. These are used in
Section 4, which gathers the proofs of the main theorems of the paper, using in particular constant term identities similar to those appearing in shuffle algebras. In Section 5, we introduce furhter currents 
using automorphisms of $\U_r$ and obtain an embedding of $\U_r$ into a quotient of the full quantized affine algebra of 
$sl_2$ with non-standard vanishing conditions on the Cartan currents. We gather discussions and concluding remarks in Section 6, in particular we address the full $(q,t)$-deformation of the constructions of this paper.

\vskip.2in

\noindent{\bf Acknowledgments.} We thank O.Babelon, F. Bergeron, J.-E. Bourgine, I. Cherednik, A. Negut, V. Pasquier, and O. Schiffmann for discussions at various stages of this work. R.K.'s research is supported by NSF grant DMS-1404988. P.D.F. is supported by the NSF grant DMS-1301636 and the Morris and Gertrude Fine endowment. R.K. would like to thank the Institut de Physique Th\'eorique (IPhT) of Saclay, France, for hospitality during various stages of this work. The authors also
acknowledge hospitality and support from Galileo Galilei Institute, Florence, Italy, as part of the scientific program on
``Statistical Mechanics, Integrability and Combinatorics", from the Centre de Recherche Math\'ematique de l'Universit\'e de Montreal during the thematic semester: ``AdS/CFT, Holography, Integrability",
as well as of the Kavli Institute for Theoretical Physics, Santa Barbara, California,  during the program ``New approaches to non-equilibrium and random systems", supported by the NSF grant PHY11-25915.

\section{An extended quantum Q system}
Let us first define the algebra $\U_r$. The definition is via generators and relations,
where the relations are what we call the $M$-system, commutation relations and a rank condition. We then state a new definition, using a restricted set of generators, and a different set of relations. The main relation takes the familiar form of an exchange relation in a current algebra. One of the main theorems in this paper is that the two definitions are equivalent. 

The $M$-system is the mutation relation which appears in the quantum cluster algebra defined by $Q$-system of type $A_r$ \cite{DFKnoncom,qKR}, extended by a single central element which is a coefficient of the cluster algebra, as well as a degree operator. 

The alternative definition given in this paper is as a rank-dependent quotient of a current algebra. The current algebra algebra is identified with the Drinfeld quantization of the loop algebra of the nilpotent subalgebra of $\sl_2$, a subalgebra of the quantum affine algebra \cite{Drinfeld}. The generators of $\U_r$ are either components of the generating function of this subalgebra, or polynomials in these generators, given by a quantum determinant formula. 

\subsection{The algebra $\U_r$}

\begin{defn} Let $\U'$ be the algebra generated over the ring $\Z_q=\Z[q,q^{-1}]$ by the non-commuting generators
$$\{M_{\al,n}, \al\in \N, n\in\Z\},$$ 
subject to two sets of relations. 
The first is called the $M$-system\footnote{Throughout this paper, we shall refer to equation \eqref{Msys} 
as the ``$M$-system" relation, in analogy
with the $Q$-system relation (see below for a precise connection).}:
\begin{equation}
q^{\al} M_{\al,n+1} M_{\al,n-1} = M_{\al,n}^2 - M_{\al+1,n}M_{\al-1,n},\quad \al\in \N, n\in \Z,
\label{Msys}
\end{equation}
with the convention that $M_{0,n}=1$ for all $n$. The second set of relations are commutation relations among the generators:
\begin{equation}
M_{\al,n} M_{\beta,n+\epsilon} = q^{\min(\al,\beta)\epsilon} M_{\beta,n+\epsilon}M_{\al,n}, \quad \epsilon\in\{0, 1\}, \al,\beta \in \N, n\in\Z.\label{Mcom} 
\end{equation}
\end{defn}



Finally, there is a rank-dependent relation.
\begin{defn} Let $r$ be a fixed positive integer. The
algebra $\U'_r$ is the quotient of $\U'$ by the ideal generated by the relations
\begin{equation}\label{Ur}
M_{r+2,n}=0 , \quad n\in \Z.
\end{equation}
\end{defn}


\begin{lemma}
The algebras $\U'$ and $\U_r'$ are $\Z$-graded, with
$$
\deg M_{\al,n} = \al n.
$$
\end{lemma}
\begin{proof}
We rewrite the relation \eqref{Msys} as
\begin{equation}
M_{\al+1,n} M_{\al-1,n} = M_{\al,n}^2 - q^\al M_{\al,n+1} M_{\al,n-1}.
\end{equation}
This is a recursion relation on the index $\al$, starting with the collection $M_{0,n}=1$ which have degree 0, and $M_{1,n}$, which have degree $n$. The statement for $\U'$ follows by induction, since the  commutation relations \eqref{Mcom} and the rank restriction \eqref{Ur} are homogeneous relations.
\end{proof}

We adjoin the invertible degree operator $\Delta$ to $\U'$ and $\U_r'$, where
$$
\Delta x = q^n x \Delta, \quad x\in \U'[n],
$$
where $\U'[n]$ is the homogeneous graded component of $\U'$, and similarly for $\U'_r$.
That is, 
\begin{equation}\label{DeltaM}
\Delta M_{\al,n} = q^{\al n} M_{\al,n} \Delta.
\end{equation}
The algebras extended by $\Delta^{\pm1}$ are denoted by $\U$ and $\U_r$, respectively.

\subsection{Relation to the quantum $Q$-system}
The algebra $\U_r^{\rm loc}$ is closely related to the $A_r$ quantum $Q$-system introduced in \cite{DFKnoncom,qKR}. The quantum $Q$-system can be considered in the language of quantum cluster algebras\footnote{In cluster algebras, all inverses of the generators are adjoined by definition, hence the localization.}. In the current paper
we add a single, central coefficient $A$ to the cluster algebra of the quantum $Q$-system. The quantum cluster algebra includes the inverses of all the generators, by definition. With the addition of $\Delta$, the solutions of the quantum $Q$-system generate an algebra which is the localization of $\U_r$.

The quantum $Q$-system is a subalgebra of a quantum cluster algebra with trivial coefficients\footnote{To be precise, to compare with the usual definition of \cite{BernZel}, the generators are renormalized cluster variables}. We add a frozen variable $A$ as in the quiver illustrated in Figure \ref{fig:quiver}.
The ring over which the algebra is defined is $\Z_v=\Z[v,v^{-1}]$ where $v=q^{-1/(r+1)}$. We also adjoin the $(r+1)$-st root of the degree operator $\Delta$.
\begin{figure}
\centering
\includegraphics[width=12.cm]{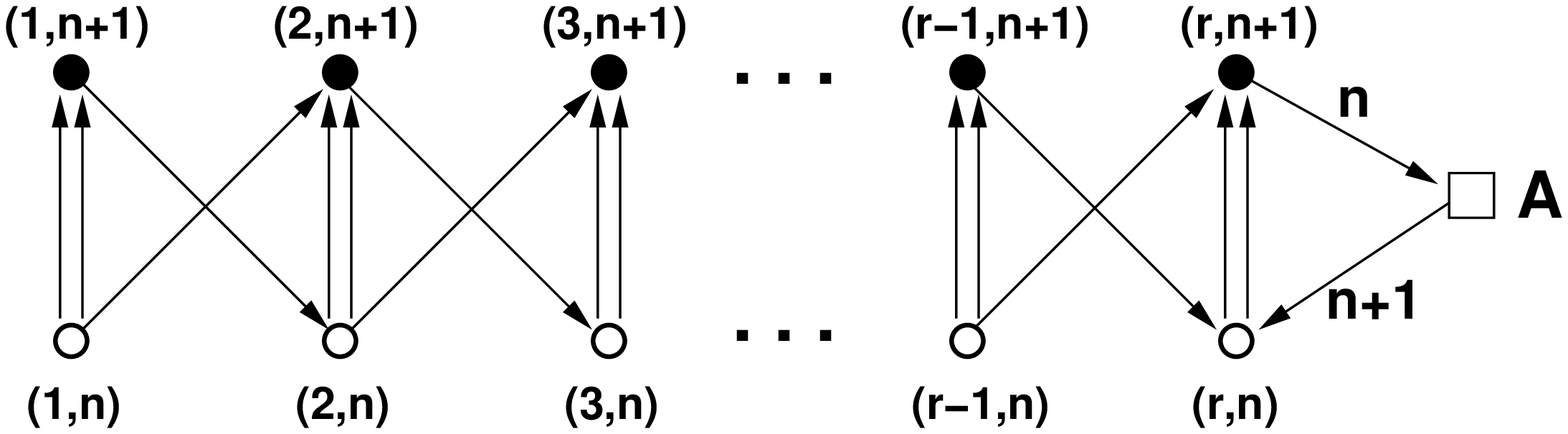}
\caption{\small }
\label{fig:quiver}
\end{figure}

Let $C$ be the Cartan matrix of the Lie algebra $\sl_{r+1}$, and define the integer matrix $\Lambda = (r+1) C^{-1}$. Explicitly,
\begin{equation}\label{lambda}
\Lambda_{\al,\beta} = \min(\al,\beta)(r+1-\max(\al,\beta)).
\end{equation}

Define the generators
\begin{equation}\label{QfromM}
\Q_{\al,n} = v^{-d_{\al}}M_{\al,n} \Delta^{-\frac{\al}{r+1}},
\quad  \al\in [0,r+1],n\in\Z
\end{equation}
where 
$$
d_{\al}=\frac{\Lambda_{\al,\al}}{2} n+ \sum_{\beta}\Lambda_{\al,\beta}=\frac{\Lambda_{\al,\al}}{2}(n+r+1).
$$
We use the convention $\Lambda_{0,\beta}=\Lambda_{r+1,\beta}=0$ for all $\beta$ so that $\Q_{0,n}=1$ and $\Q_{r+2,n}=0$ for all $n$.

\begin{lemma}
Relations \eqref{Msys}, \eqref{Mcom} and \eqref{Ur} imply that the generators $\Q_{\al,n}$ satisfy the relations
\begin{eqnarray}
v^{-\Lambda_{\al,\al}} \Q_{\al,n+1} \Q_{\al,n-1} &=& \Q_{\al,n}^2 - q\Q_{\al+1,n}\Q_{\al-1,n},\quad \al\in[1,r],n\in\Z\label{Qsys}\\
\Q_{\al,n} \Q_{\beta,n+1} &=&  v^{\Lambda_{\al,\beta}}\Q_{\beta,n+1}\Q_{\al,n}, \quad \al,\beta\in [0,r+1], n\in \Z,\nonumber\label{Qcommutation}\\
\Q_{r+2,n}&=&0. \nonumber \label{Qboundary}
\end{eqnarray}
\end{lemma}

We identify the central element $A$ as the frozen variable in the quantum cluster algebra as in Figure \ref{fig:quiver}. Noting that this quiver corresponds to the cluster variables 
$$
\bx_n=(\Q_{1,n},\cdots,\Q_{r,n}| \Q_{1,n+1},\cdots,\Q_{r,n+1} |A),$$
  we identify $\Q_{r+1,n}=A^n$ in Equations \eqref{Qsys}, each of which is then identified (up to a normalization of the cluster variables) with a mutation in the quantum cluster algebra. 

\begin{remark}
The boundary condition $\Q_{r+1,n}=A^n$ is a generalization of the boundary condition $\Q_{r+1,n}=1$ used in \cite{qKR}. Both are consistent with the less restrictive boundary condition $\Q_{r+2,n}=0$. The former corresponds to the choice $\Q_{r+1,0}=1$ and $\Q_{r+1,1}=A$. The introduction of a coefficient $A$ is necessary in this paper in order for the algebra to have a $\Z$-grading.
\end{remark}

Apart from the frozen variable $A$, this system was called the $A_r$ quantum Q-system in \cite{DFKnoncom,qKR}. 
It is the quantum version of the $A_r$ Q-system cluster algebra identified in \cite{Ke07}.

The quantum cluster algebra structure allows to extend straightforwardly the commutation relations
\eqref{Mcom} to the following:
\begin{lemma}\label{farcommutation}
\begin{equation}\label{comM}
M_{\al,n} M_{\beta,n+p} = q^{p\min(\al,\beta)} M_{\beta,n+p}M_{\al,n}, \quad n,p\in\Z,\ |p|\leq |\beta-\al|+1.
\end{equation}
\end{lemma}
\begin{proof}
The proof of the Lemma for the variables $Q$ is given in Lemma 3.2 of \cite{qKR}. Then using the relation of $M$ and $Q$ variables gives \eqref{comM} directly. The condition on $p$ is that the corresponding quantum cluster variables belong to a common cluster. The appearance of the quantity $\min(\al,\beta)$ is due to the relation
$\Lambda_{\al,\beta}+\al \beta=(r+1)\min(\al,\beta)$.
\end{proof}



The $\Z$-grading of the algebra $\U_r$ is inherited by the quantum $Q$-system variables, with $\deg \Q_{\al,n}=\al n$. Since $\Q_{r+1,n}=A^n$, we have $\deg A = r+1$, namely
\begin{equation}\label{degA}
\Delta\, A= q^{r+1} A\, \Delta ,\end{equation}
and
\begin{equation}\label{bcm}
M_{r+1,n} = A^n \Delta.
\end{equation}

As a consequence of \eqref{DeltaM} and the fact that $\Q_{r+1,n}$ commutes with all $\Q_{\al,n}$,
\begin{equation}
A\, M_{\al,n} =q^{-\al} M_{\al,n}\, A \label{AMcom}
\end{equation}

We will work with the generators $M_{\al,n}$ in this paper instead of $\Q_{\al,n}$ for convenience 
(they generate the same algebra). 
In what follows, we will consider the localization of this algebra obtained by adjoining inverses of the generators.

\subsection{Generating functions}
We want to present the relation of the algebra $U$ to a loop algebra, so we introduce generating functions, or currents.
 Denote the generating functions
\begin{equation}\label{mcurrents}
{\mathfrak m}_\al(z) := \sum_{n\in\Z} M_{\al,n} z^n.
\end{equation}
The subset of generators $\{M_n:=M_{1,n}: n\in \Z\}$ plays a special role in the presentation which follows, and we introduce the notation
\begin{equation}\label{mcurrent}
\m(z) := \m_1(z).
\end{equation}

Equations \eqref{Msys}, \eqref{Mcom} and \eqref{Ur} in terms of currents as follows.
Let $\delta(z)$ denote the distribution
\begin{equation}\label{delta}
\delta(z):=\sum_{n\in\Z}z^n=\frac{z^{-1}}{1-z^{-1}}+\frac{1}{1-z} 
\end{equation}
where the first fraction is understood as a power series in $z^{-1}$ and the second fraction is expanded in powers of $z$\footnote{
This is a delta function in the sense that
$\oint \frac{du}{2i\pi u} f(u)\delta(u/z)=f(z)$, where the contour integral
picks out the constant term of the current $f(u)\delta(u/z)$.}.

Define the ``constant term" notation
\begin{equation}\label{CTdef}
CT_{u_1,...,u_n}\left(f(u_1,...,u_n)\right)=\oint \frac{du_1}{2i\pi u_1}\cdots \oint \frac{du_n}{2i\pi u_n}
f(u_1,...,u_n)\end{equation}
for the multiple constant term in variables $u_1,...,u_n$.

 In terms of generating functions,
the $M$-system equations \eqref{Msys} are equivalent to the relations
\begin{equation}\label{Msysprime}
CT_{u,v}\left( \left(\left(1-q^\al \frac{v}{u}\right){\mathfrak m}_\al(u){\mathfrak m}_\al(v)
-{\mathfrak m}_{\al+1}(u){\mathfrak m}_{\al-1}(v)\right)\delta(u v/z)\right)=0,\quad 1\leq \al\leq r+1,
\end{equation}
in that the coefficient of $z^n$ in this equation, for each $n$ and $\al$, gives one of the relations \eqref{Msys}. 

Similarly, the commutation relations \eqref{Mcom} are components of the relation
\begin{equation}\label{Mcomprime}
CT_{u,v}\left(\left(\frac{1}{v^{\epsilon}}{\mathfrak m}_\al(u){\mathfrak m}_\beta(v)
-q^{\epsilon{\rm Min}(\al,\beta)}\frac{1}{u^{\epsilon}}{\mathfrak m}_\beta(u){\mathfrak m}_\al(v)\right)\delta(u v/z)\right)=0,
\quad |\epsilon|\leq 1.
\end{equation}
The grading relation \eqref{DeltaM} is equivalent to
$$
\Delta\, {\mathfrak m}_\al(z)\ =\ {\mathfrak m}_\al(q^{\al}z) \, \Delta, \qquad
(1\leq \al\leq r+1).
$$

\subsection{Alternative formulation of $\U_r$}
The main claim of this paper is that there is an alternative definition of the algebra $\U_r$ in terms of the subset of generators $M_{n}:=M_{1,n}$ or the currents $\m(z)$ and relations on them. We present in this subsection some of the main theorems to be proven.

Define the $q$-commutators
$$
[x,y]_q := xy-q yx, \quad x,y \in U.
$$

\begin{thm}
If $M_{n}$ are generators which satisfy both the $M$-system \eqref{Msys} and the commutation relations \eqref{Mcom}, then they satisfy the following exchange relations:
\begin{equation}\label{quadratic}
[M_{n},M_{n+p}]_q + [M_{n+p-1},M_{n+1}]_q = 0, \qquad n\in \Z, 1\leq p.
\end{equation}
\end{thm}
Note that this is a relation for the subset of generators $M_n=M_{1,n}$ only and does not involve $M_{\al,n}$ with $\al>1$.

For example, if $p=1$, this relation reduces to the commutation relation between $M_{1,n}$ and $M_{1,n+1}$. If $p=2$, this is a difference of \eqref{Msys} and its conjugate version 
 \eqref{leftM}, at shifted $n$.

The term exchange relation is justified because in terms of currents, Equation \eqref{quadratic} takes on the more familiar form
\begin{equation}
\label{current_quadratic}
(z-q w) \m(z) \m(w) + (w-qz) \m(w) m(z) = 0.
\end{equation}
Each of the relations \eqref{quadratic} is a coefficient of $z^{n-1} w^{n+p}$ of the relation \eqref{current_quadratic} for each $n,p\in \Z$.

Equation \eqref{current_quadratic} is the relation satisfied by the Drinfeld generators \cite{Drinfeld} of the nilpotent subalgebra $U_{\sqrt{q}}(\n[u,u^{-1}])$ of the quantum affine algebra $U_{\sqrt{q}}(\widehat{\sl}_2)$.
We note that this relation also appears in the context of the spherical Hall algebra \cite{spherical_hall}.

\subsection{Polynomiality of $M_{\al,n}$ in the generators}

The second important result is that all the generators which do not appear as part of the exchange relation \eqref{quadratic} are in the universal enveloping algebra generated by those that do.
The generators $M_{\al,n}$ with $\al>1$, which are a priori defined from \eqref{Msys}
 in terms of a localization of the algebra generated by components of the current $\m(z)$, are, in fact, polynomials in these generators. 
 
\begin{thm}
If the elements $M_{\al,n}$ satisfy the $M$-system \eqref{Msys} and the commutation relations \eqref{Mcom}, the following recursion relations are satisfied:
\begin{equation}\label{qcommInd}
(-1)^\al (q-1) M_{\al,n}=[M_{n-\al+1},M_{\al-1,n+1}]_{q^\al}, \qquad \al>1.
\end{equation}
\end{thm}

In terms of currents, these recursion relations can be expressed as
\begin{equation}\label{alpharec}
\qquad {\mathfrak m}_\al(z)=CT_{u_1,u_2}\left(
\frac{u_1^{\al-1}u_2^{-1}{\mathfrak m}(u_1){\mathfrak m}_{\al-1}(u_2)
-q^\al u_1^{-1} u_2^{\al-1}{\mathfrak m}_{\al-1}(u_1){\mathfrak m}(u_2)}{(-1)^\al (q-1) }\, \delta(u_1u_2/z)\right).
\end{equation}

Moreover, we have the following
\begin{cor}
Equation \eqref{qcommInd} implies a polynomial expression for $M_{\al,n}$ in terms of the components of the current $\m(z)$:
\begin{equation}\label{firstMalpha}
(q-1)^{\al-1}M_{\al,n} = (-1)^{\al(\al-1)/2}[ \cdots [M_{n-\al+1},M_{n-\al+3}]_{q^2},M_{n-\al+5}]_{q^3}, \cdots, M_{n+\al-1}]_{q^\al}.
\end{equation}
\end{cor}

Therefore, the algebra $U$ is generated by the components of the single current $\m(z)$, all other generators being polynomials in these generators. 
The algebra $U_r$ is a quotient of this algebra.

It would seem that the polynomial expression for $M_{\al,n}$ of \eqref{firstMalpha} in the $M_n$'s has
coefficients in $\C(q)$, due to the denominators involving $(q-1)$. It turns out
that $M_{\al,n}$ also has a polynomial expression of the $M_n$'s, with coefficients in $\Z[q]$.

Indeed, modulo the exchange relations \eqref{quadratic}, the nested commutator expression \eqref{firstMalpha}
is equivalent to the following. In terms of the currents ${\mathfrak m}_\al(z)$,
and the slightly non-standard definition of $q$-vandermonde product:
\begin{equation}\label{qvander}
\Delta_q(u_1,...,u_\al)=\prod_{1\leq a<b\leq \al} \left(1-q\frac{u_b}{u_a}\right) ,
\end{equation}
we have:
\begin{thm}\label{vanderq}
The current ${\mathfrak m}_\al(z)$ is expressed in terms of the current ${\mathfrak m}(z)$ via the following 
``quantum determinant" constant term identity:
\begin{equation}\label{qdetone}
{\mathfrak m}_\al(z)=CT_{u_1,\cdots,u_\al}\left( \Delta_q(u_1,...,u_\al) 
\,\prod_{i=1}^\al {\mathfrak m}(u_i)\, \delta(u_1...u_\al/z)\right)\ .
\end{equation}
\end{thm}

Section 4 of this paper is devoted to the proof of the equivalence of the two presentations of the algebra $U_r$.
\begin{thm}\label{Main}
The algebra $U_r$ generated by $\{M_{\al,n}, \al\in 1,r+1, n\in \Z\}$ subject to the $M$-system relations \eqref{Msys}, the commutation relations \eqref{Mcom} and the rank restriction \eqref{Ur}, is isomorphic to the algebra generated by $\{M_n, n\in\Z\}$ subject to the exchange relations \eqref{current_quadratic} and the rank restriction \eqref{Ur}, expressed by using 
\eqref{qdetone} as definition of the $M_{\al,n}$'s.
\end{thm}

That is, $U_r$ is isomorphic to the quotient of $U_{\sqrt{q}}(\n_+[u,u^{-1}])$ by the two-sided ideal generated by the relation \eqref{Ur}.


The proof relies on the existence of commuting Hamiltonians of the $M$-system (conserved quantities of the $Q$-system). We have previously derived these in the quotient $\U_r$ only and they are used in the derivation of the result. However, the resulting relations are independent of $r$ and hold in the algebra $\U$ as well.

\section{Conserved quantities}
In this section, all relations are considered in the localization $\U_r^{\rm loc}$.

We recall that the quantum $Q$-system can be regarded as a discrete integrable equation with conserved quantities which are Laurent polynomials in the generators \cite{DFKnoncom}.

Thus, the relations \eqref{Msys}, regarded as relations in $\U_r^{\rm loc}$, can be considered to be a discrete evolution of the variables $M_{\al,n}$ in the discrete ``time" variable $n$. When regarded in this way, the equation is a discrete integrable system, and there are $r$ algebraically independent elements of $U_r^{\rm loc}$, Laurent polynomials in any initial data, which are independent of $n$. Here, we reformulate the discrete integrable structure 
first established for the classical $Q$-system in \cite{DFK10} and then extended to the quantum case in \cite{DFKnoncom}.

\subsection{Miura operator and quantum conserved quantities}

Let $D$ denote the automorphism of $\U_r$ acting as single time step shift operator,
\begin{equation}\label{defD}
D (M_{\al,n}) = M_{\al,n+1},\quad \al\in [1,r+1] ,
\end{equation}
extended to an action on $\U_r^{\rm loc}$.  Define the monomials
$$
\xi_{\al,n} = M_{\al,n} M_{\al,n+1}^{-1}, \qquad x_{\al,n} = \xi_{\al,n}^{-1} \xi_{\al-1,n}.
$$
In particular, $\xi_{0,n}=1$ and $\xi_{r+1,n} = A^{-1}.$ Consider the following element, acting on $\U^{\rm loc}_r$:
\begin{equation}\label{miura}
\mu_n =( D-x_{r+1,n}) (D-x_{r,n}) \cdots (D-x_{1,n}).
\end{equation}
We call this the Miura operator.
We claim that $\mu_n=\mu_{n-1}$ and hence is independent of $n$. This proved using the following Lemma, which uses M-system \eqref{Msys} and the commutation relations \eqref{Mcom} (but not the rank restrictions \eqref{Ur}):
\begin{lemma}\label{ziplemma}
For all $\al\in [1,r]$ and $n\in \Z$,
\begin{equation}\label{zipper}
(D-x_{\al+1,n})(D-\xi_{\al,n}^{-1}\xi_{\al-1,n-1})= (D-\xi_{\al+1,n}^{-1}\xi_{\al,n-1})(D-x_{\al,n-1}).
\end{equation}
\end{lemma}
\begin{proof}
This is a quadratic equation in $D$. The coefficients of $D^2$ on both sides of \eqref{zipper}
are 1, and the coefficients of $D^0$ are equal to $\xi_{\al+1,n}^{-1} \xi_{\al,n-1}$ 
on both sides by definition of $x_{\al,n}$.

Identifying the coefficients of $D$ is equivalent to the following equation:
$$
\xi_{\al+1,n}^{-1}\left( \xi_{\al,n} - \xi_{\al,n-1}\right)=\left(\xi_{\al,n}^{-1}-\xi_{\al,n+1}^{-1}  \right)\xi_{\al-1,n}.
$$
This follows from \eqref{Msys}, which implies the ``left" M-system:
\begin{equation}\label{leftM}
q^{-\al} M_{\al,n-1}M_{\al,n+1} = M_{\al,n}^2-q^{-1} M_{\al+1,n}M_{\al-1,n}.
\end{equation}
obtained by multiplying \eqref{Msys} from the left by $M_{\al,n-1}$
and from the right by $M_{\al,n-1}^{-1}$, and further simplifying by use of the commutation relations \eqref{Mcom}.
Multiplying \eqref{leftM} from the left by $M_{\al,n}^{-1}$ and from the right by $M_{\al,n+1}^{-1}$, 
using the commutation relations \eqref{Mcom} and rearranging terms, this is equivalent to:
\begin{equation}\label{xione}
\xi_{\al,n}-\xi_{\al,n-1} = q^{-1} M_{\al+1,n} M_{\al,n}^{-1} M_{\al-1,n} M_{\al,n+1}^{-1}.
\end{equation}
On the other hand, the equation \eqref{Msys} can be written as
$$
M_{\al,n+1}^2-q^\al M_{\al,n+2}M_{\al,n} = M_{\al+1,n+1}M_{\al-1,n+1}.
$$
Multiplying from the left by $M_{\al,n+1}^{-1}$ and from the right by $M_{\al,n}^{-1}$, and using the commutation relations, this is:
\begin{equation}\label{xitwo}
\xi^{-1}_{\al,n}-\xi^{-1}_{\al,n+1} = M_{\al,n+1}^{-1} M_{\al+1,n+1} M_{\al-1,n} M_{\al,n}^{-1}.
\end{equation}
Finally, the desired equality is obtained by multiplying \eqref{xione} from the left by 
$\xi_{\al+1,n}^{-1}=M_{\al+1,n+1}M_{\al+1,n}^{-1}$, and \eqref{xitwo} from the right by 
$\xi_{\al-1,n}=M_{\al-1,n}M_{\al-1,n+1}^{-1}$, and using the commutation relations to see that they are equal.
\end{proof}

\begin{thm} The operator $\mu_n$ is independent of $n$. That is, $\mu_n=\mu_{n-1}=\mu$.
\end{thm}
\begin{proof}
Using Lemma \ref{ziplemma}, we have a ``zipper argument" as follows. Start with
\begin{eqnarray*}
(D-x_{2,n})(D-x_{1,n}) &=& (D-x_{2,n})(D-\xi_{1,n}^{-1}\xi_{0,n})\\ &=&
(D-x_{2,n})(D-\xi_{1,n}^{-1}\xi_{0,n-1}) \\ &=& (D-\xi_{2,n}^{-1}\xi_{1,n-1})(D-x_{1,n-1}).
\end{eqnarray*}
Here, we made use of the fact that $\xi_{0,n}=\xi_{0,n-1}=1$ and applied Lemma \ref{ziplemma}. 
We then use Lemma \ref{ziplemma} iteratively to update each of the subsequent terms, and for any fixed positive $r$, this ends with the updated term:
\begin{eqnarray*}
(D-\xi_{r+1,n}^{-1}\xi_{r,n-1})=(D-\xi_{r+1,n-1}^{-1}\xi_{r,n-1})=(D-x_{r,n-1}),
\end{eqnarray*}
because $\xi_{r+1,n} = A^{-1}=\xi_{r+1,n-1}.$
Therefore, $\mu_n=\mu_{n-1}$.
\end{proof}

\begin{example}\label{exAone}
In the case of $A_1$, $r=1$, we have $\xi_n:=\xi_{1,n}=M_nM_{n+1}^{-1}$, 
$\xi_{2,n}=M_{2,n}M_{2,n+1}^{-1}=A^{-1}$, $x_{1,n}=\xi_n^{-1}$, $x_{2,n}=A \xi_n$, hence:
$$\mu=(D-A\xi_n)( D-\xi_n^{-1})=D^2-C_1 D+A$$
where $C_1=A \xi_n+\xi_{n+1}^{-1}$.
\end{example}

Note the following simple commutation relations.
\begin{lemma}\label{commlem} We have the commutation relations:
\begin{eqnarray}
\xi_{\al,n}\, \xi_{\beta,n+p}&=&\xi_{\beta,n+p}\, \xi_{\al,n}\qquad (|p|\leq |\al-\beta|)\label{xico}\\
x_{\al,n}\, x_{\beta,n+p}&=&x_{\beta,n+p}\, x_{\al,n}\qquad (|p|\leq {\rm Max}(|\al-\beta|-1,0))\label{xco}\\
M_n \,\xi_{\beta,n+p}&=&q^{-1}\xi_{\beta,n+p}\, M_n\qquad (-\beta\leq p\leq \beta-1)\label{comxi}
\end{eqnarray}
\end{lemma}
\begin{proof}
Using \eqref{comM}, we deduce easily that for $\al<\beta$ and $|p|\leq \beta-\al$:
\begin{eqnarray*}\xi_{\al,n}\, \xi_{\beta,n+p}&=&M_{\al,n}M_{\al,n+1}^{-1}M_{\beta,n+p}M_{\beta,n+p+1}^{-1}\\
&=&
q^{\al(p-(p+1)-(p-1)+p)}M_{\beta,n+p}M_{\beta,n+p+1}^{-1}M_{\al,n}M_{\al,n+1}^{-1}=\xi_{\beta,n+p}\,
\xi_{\al,n}
\end{eqnarray*}
which implies \eqref{xico}. The relation for $x_{\al,n}=\xi_{\al,n}^{-1}\xi_{\al-1,n}$ follows immediately.
Finally \eqref{comxi} follows from \eqref{comM}, while the range of validity corresponds to 
both $|p|\leq \beta$ and $|p+1|\leq \beta$.
\end{proof}

The conservation of the Miura operator $\mu_n=\mu$ is best expressed via the expansion
\begin{equation}\label{expanmiu}
\mu_n=\mu=\sum_{m=0}^{r+1} (-1)^m C_m\, D^{r+1-m}=\prod_{\al=r+1}^1 (D-x_{\al,n})
\end{equation}
The conserved quantities $C_m$ are expressed as homogeneous polynomials of the $x$'s of total degree $m$.
Since the Miura operator \eqref{miura} is independent of $n$, we can write it in terms of the variables in the limit when $n$ 
tends to infinity. In \cite{qKR} (Lemma 5.14), we showed that the limit $\Q_{\al,n}\Q_{\al,n+1}^{-1}$ exists when $n\to\infty$. Therefore, so do the limits
$$
\xi_\al:=\lim_{n\to\infty} \xi_{\al,n}, \quad x_\al := \lim_{n\to\infty} x_{\al,n}\ .
$$
Moreover, in the limit, the variables $\{x_1,...,x_r\}$ commute among themselves as a consequence 
of Lemma \ref{commlem}
and commute with $D$ as well. On the other hand, since $\mu_n$ is independent of $n$, we have
$$
\mu =\lim_{n\to\infty} \mu_n = (D-x_{r+1}) \cdots (D-x_1).
$$
Comparing this with the expansion $\mu=\sum_{m=0}^{r+1}(-1)^m C_{m}\, D^{r+1-m}$, we can identify $C_m$ with the elementary symmetric functions in the commuting variables $\{x_\al\}$:
\begin{equation}\label{elementary}
C_{m} = e_{m}(x_1,...,x_{r+1}).
\end{equation}
In particular, we have
\begin{lemma}\label{commutaC} The conserved quantities commute among themselves:
\begin{equation}
C_m\, C_p=C_p\, C_m\qquad (m,p=0,1,...,r+1)\label{integC}.
\end{equation}
\end{lemma}

Finally, note that 
$C_m$ has degree $m$ with respect to the $\Z$-grading of $\U_r$, 
as $\xi_{\al,n}$ all have degree $-\al$ and $x_{\al,n}$, hence $x_\al$, all have degree $1$.
That is,
\begin{equation}
\Delta\, C_m = q^m\, C_m \, \Delta \quad (m=0,1,...,r+1)\label{comCdelt}.
\end{equation}



\subsection{A linear recursion relation}

When the Miura operator acts on $\U_r$ by left-multiplication, it gives rise to a linear recursion relation among the generators, which is a characteristic of integrability.
\begin{lemma} \label{linreclemma}
The generators $M_n := M_{1,n}$ satisfy a linear recursion relation with coefficients which are independent of $n$:
\begin{equation}\label{consar}
\sum_{j=0}^{r+1} (-1)^j C_j \, M_{m-j} = 0.
\end{equation}
\end{lemma}
\begin{proof}
Consider  $\mu(M_n)=\mu_n(M_n)$. Noting that the rightmost factor of $\mu_n$ is $(D-x_{1,n})= D-\xi_{1,n}^{-1} = D-M_{n+1}M_n^{-1}$, we conclude that $\mu(M_n)=0$.
\end{proof}

\begin{remark}
The construction of our Miura operator is very similar to the so-called Miura transformation allowing for defining generators of the q-deformed W-algebra in terms of quantum group generators \cite{FR95} (See e.g. Eq. (10.9).) . 
\end{remark}

\subsection{Conserved quantities as Hamiltonians}

\begin{figure}
\includegraphics[width=8.cm]{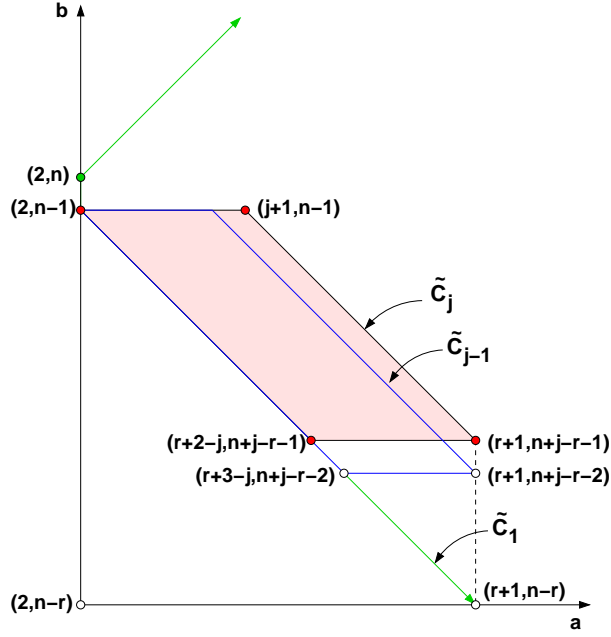}
\caption{\small The domains $\tilde{\rm C}_j$ and $\tilde{\rm C}_{j-1}$ of indices ot the $x$'s that contribute to
$\tilde{C}_{j,n+j-r-1}$ and $\tilde{C}_{j-1,n+j-r-1}$ respectively. The larger wedge corresponds
to indices $a=\al$, $b=n+p$ such that $x_{a,b}$ commutes with $M_n$, and contains the former.}\label{fig:rangecomm}
\end{figure}

We now consider commutations relations in the algebra which we interpret as the integrable structure. 
The conserved quantities $\{C_m\}$ play the role of Hamiltonians in $\U_r$, in the sense that they act by time translation on the generators $\{M_n=M_{1,n}, n\in\Z\}$.
\begin{thm}\label{timeTranslation}
Let $m\in [1,r]$. Then 
\begin{equation}\label{commuCm}
[C_m, M_{n}] = (q-1) \sum_{j=1}^m (-1)^j C_{m-j}M_{n+j}.
\end{equation}
\end{thm}
\begin{proof}
Recall the Miura operator, $\mu_n  = \prod_{r+1}^1(D-x_{\al,n})$ and define $\widetilde{\mu}_n=\prod_{r+1}^2 (D-x_{\al,n}).$
We define the coefficients $\widetilde{C}_{j,n}$ by expanding $\widetilde{\mu}_n$,
$$
\widetilde{\mu}_n = \sum_{j=0}^{r} (-1)^j \widetilde{C}_{j,n} D^{r-j}, \quad \widetilde{C}_{0,n}=1.
$$
Multiplying this formula for  $\widetilde{\mu}_n$ by $(D-\xi_{1,n}^{-1})$ and identifying coefficients of $D^{r+1-j}$,
\begin{equation}\label{recuCm}
C_j=\widetilde{C}_{j,n}+\widetilde{C}_{j-1,n}\xi_{1,n+r+1-j}^{-1}=\widetilde{C}_{j,n+j-r-1}+\widetilde{C}_{j-1,n+j-r-1}\xi_{1,n}^{-1}
\end{equation}
where the last equality follows from the fact that $C_j$ is independent of $n$. 
We use the following commutation relations, easily derived from \eqref{comM}:
\begin{eqnarray*}
M_n\xi_{\al,n+p}&=& q^{-1} \xi_{\al,n+p} M_n,\quad -\al \leq p \leq \al-1;\\
M_n x_{\al,n+p} &=& x_{\al,n+p} M_n, \quad \al>1 \hbox{ and } -\al+1\leq p \leq \al-2,
\end{eqnarray*}
to show that $[M_n,\tilde{C}_{j,n+j-r-1}]=[M_n,\tilde{C}_{j-1,n+j-r-1}]=0$. Indeed, by definition,
both  $\tilde{C}_{j,n+j-r-1}$ and $\tilde{C}_{j-1,n+j-r-1}$
are sums of products of $x_{\al,n+p}$'s, with $-\al+1\leq p \leq \al-2$. We have represented 
in Fig.~\ref{fig:rangecomm} the actual domains
$\tilde{\rm C}_j$ and $\tilde{\rm C}_{j-1}$ of the integer plane $(a,b)\in \Z^2$ such that $x_{a,b}$ appears in 
$\tilde{C}_{j,n+j-r-1}$ and $\tilde{C}_{j-1,n+j-r-1}$ respectively. The domain $-\al+1\leq p \leq \al-2$, for
$a=\al$, $b=n+p$ corresponds to the larger wedge, that contains both domains.

Finally, as
$M_n x_{1,n}=M_n \xi_{1,n}^{-1}=q \xi_{1,n}^{-1}M_n=q x_{1,n}M_n$, we deduce that:
\begin{eqnarray*}[C_j,M_n]&=&[C_{j,n+j-r-1},M_n]=[\tilde{C}_{j-1,n+j-r-1}\, \xi_{1,n}^{-1},M_n]\\
&=&(1-q)\tilde{C}_{j-1,n+j-r-1}\, \xi_{1,n}^{-1} M_n
=(1-q)\tilde{C}_{j-1,n+j-r-1}\, M_{n+1}
\end{eqnarray*}

Next, we use \eqref{recuCm} to express $\tilde{C}_{j-1,n+j-r-1}$ 
in terms of the $\xi$'s and the $C$'s, by rewriting iteratively:
$$\tilde{C}_{j-1,n+j-r-1}=C_{j-1}-\tilde{C}_{j-2,n+j-r-1}\, \xi_{1,n+1}^{-1}=\cdots =\sum_{k=1}^j (-1)^{k-1}C_{j-k}
\xi_{1,n+k}^{-1}\cdots \xi_{1,n+1}^{-1}$$
Noting that $\xi_{1,n+k}^{-1}\cdots \xi_{1,n+1}^{-1}=M_{n+k}M_{n+1}^{-1}$ yields \eqref{commuCm}.
\end{proof}

The result of Theorem \ref{timeTranslation} may be easily recast in terms of currents
Introduce the generating polynomial for the conserved quantities:
\begin{equation}\label{genercons}
C(z) = \sum_{j=0}^{r+1} (-z)^m C_m = \prod_{\al=1}^{r+1}(1-z x_\al) ,
\end{equation}
where the last equality comes from the identification with the elementary symmetric functions \eqref{elementary}.

Equation \eqref{commuCm} is rewritten as 
\begin{equation}\label{Cmcom} C(z)\, {\mathfrak m}(w)=\frac{z-w}{q z-w}\, {\mathfrak m}(w)\, C(z) .
\end{equation}
We recover the linear recursion relation \eqref{consar} by taking $z=w$, so that 
\begin{equation}\label{consarcur}
C(z)\, {\mathfrak m}(z)=0 .
\end{equation}
Note the right recursion relation obtained by taking $w=qz$, which reads ${\mathfrak m}(qz)\, C(z)=0$, or in
components:
\begin{equation}\label{rrec}\sum_{j=0}^{r+1} (-q)^{-j} M_{m-j}\, C_j = 0 .
\end{equation}

\begin{cor}\label{conscor}
The conserved quantities $C_1$ and $A^{-1}C_r$ act as discrete ``time" shifts on $M_n$, 
namely for all $n\in \Z$:
\begin{eqnarray}
{[}C_1,M_n{]}&=&(1-q)M_{n+1} \label{trans}\\
{[}A^{-1}\,C_r,M_n{]}&=&(1-q^{-1}) M_{n-1}\label{Crmin}
\end{eqnarray}\end{cor}
\begin{proof}
The first equality \eqref{trans} is simply \eqref{commuCm} for $m=1$. Rewriting \eqref{commuCm} for $m=r$
by using the linear recursion relation \eqref{consar} leads to
$[C_r,M_n]=(1-q)(C_rM_n-AM_{n-1})$, as $C_{r+1}=A$. We deduce \eqref{Crmin} by 
multiplying this by $A^{-1}$ and using $A^{-1}M_n=q M_nA^{-1}$. 
\end{proof}

In the expression \eqref{genercons} for $C(z)$, the coefficients are the elementary symmetric functions in $\{x_\al\}$. 
There is a natural transformation between the elementary symmetric functions and power sum symmetric functions:
$$
\exp\left(-\sum_{k\geq 1}\frac{z^k}{k} P_k\right) = C(z),
$$
where, expressed as a function of $x_\al$, $P_k=p_k(\bx)=\sum_{\al=1}^{r+1} x^k_\al$ for $k>0$. 
The elements $P_k\in {\mathcal U}_r$ are polynomial in the conserved quantities.
\begin{thm}\label{powershift} 
The elements $P_k$ are the $k$-step time-translation Hamiltonians,
\begin{equation}\label{transpk}
[P_k,M_n] = (1-q^k) M_{n+k}, \qquad k\in \N.
\end{equation}
\end{thm}
\begin{proof}
By definition, we have $-C(z)^{-1}C'(z)=-\frac{d}{dz}\,{\rm Log}\, C(z)=\sum_{k\geq 1}z^{k-1} \,P_k=:P(z)$. 
We wish to study the commutation of this function with the current ${\mathfrak m}(w)$. Rewriting \eqref{Cmcom} as:
$$C(z)^{-1} \, {\mathfrak m}(w) \, C(z)=\frac{w-qz}{w-z}\, {\mathfrak m}(w) ,$$
we may write in two ways:
\begin{eqnarray*}
\frac{d}{dz}\left( C(z)^{-1} \, {\mathfrak m}(w) \, C(z)\right)&=&
-C(z)^{-2}C'(z)\,  {\mathfrak m}(w) \, C(z)+C(z)^{-1} \, {\mathfrak m}(w)\, C'(z)\\
&=& [P(z),C(z)^{-1} \, {\mathfrak m}(w) \, C(z)]
=\frac{w-qz}{w-z}([P(z),{\mathfrak m}(w)])\\
&=&\frac{d}{dz}\left( \frac{w-qz}{w-z}\right){\mathfrak m}(w) = \frac{(1-q)w}{(w-z)^2}\, {\mathfrak m}(w)
\end{eqnarray*}
We deduce that 
$$[P(z),{\mathfrak m}(w)]=\frac{(1-q)w}{(w-z)(w-q z)}\, {\mathfrak m}(w)
=\left(\frac{1}{w-z }-\frac{q }{w-q z }\right)\, {\mathfrak m}(w) .$$
Eqn. \eqref{transpk} is the coefficient of $z^{k-1} w^n$ in this current identity. 
\end{proof}

\subsection{Conserved quantities as quantum determinants}

In Ref.~\cite{DFK10}, it was shown that the conserved quantities of the 
classical $Q$-system (with $q=1$, $A=\Delta=1$, $Q_{\al,n}=M_{\al,n}$)
can be expressed as discrete Wronskian determinants with a defect, of the form:
$$C_m=\det_{1\leq a<b\leq r+2\atop b\neq r+2-m}(M_{n+a+b-r-2}) \qquad (m=0,1,...,r+1)$$
Note the ``defect" between the two consecutive columns $b=r+1-m$ and $b=r+3-m$.
In this section we present the quantum counterpart of this result, which we may view as a quantum
determinant.

The explicit expression for conserved quantities coming from the expansion of the Miura operator $\mu_n$
is Laurent polynomial in the $M_{\al,n}$'s. However, using \eqref{Msys}, \eqref{Mcom} and \eqref{Ur},
it is possible to rewrite $C_m$, up to some factor $\Delta^{-1}$, as a polynomial of the $M_n$'s:

\begin{thm}\label{defectqdet}
There is a constant term expression for the conserved quantities $C_m$, in terms of the currents $\m(u)$, as a quantum determinant:
\begin{equation}\label{Cmexprqdet}
C_m=CT_{u_1,...,u_{r+1}}\left(\frac{
\Delta_q(u_1,...,u_{r+1})}{u_1u_{2}\cdots u_{m}} \prod_{a=1}^{r+1}{\mathfrak m}(u_a)\right)\, \Delta^{-1}.
\end{equation}
\end{thm}

\begin{example}\label{exAonebis}
In the case of $A_1$, $r=1$, we get
\begin{eqnarray*}
C_0&=&CT_{u_1,u_2}\left( (1-q\frac{u_2}{u_1}){\mathfrak m}(u_1){\mathfrak m}(u_2) \right)\Delta^{-1}=(M_0^2-q M_1M_{-1})\Delta^{-1}=M_{2,0}\Delta^{-1}=1\\
C_1&=&CT_{u_1,u_2}\left( (\frac{1}{u_1}-q\frac{u_2}{u_1^2}){\mathfrak m}(u_1){\mathfrak m}(u_2) \right)=
(M_1M_0-q M_2 M_{-1})\Delta^{-1}\\
C_2&=&CT_{u_1,u_2}\left( (\frac{1}{u_1u_2}-q\frac{1}{u_1^2}){\mathfrak m}(u_1){\mathfrak m}(u_2) \right)=
(M_1^2-q M_2M_{0})\Delta^{-1}=M_{2,1}\Delta^{-1}=A
\end{eqnarray*}
\end{example}

The proof of Theorem \ref{defectqdet} is given in Section \ref{secdefect} below.

\section{Proofs}
\label{proofsec}
In this section, we prove the main Theorems of the paper. 
In the first two sections, we prove Eqs. \eqref{quadratic} and \eqref{qcommInd} 
directly in components, as relations between
the $M_{\al,n}$,s and $M_n$'s. The technique is by induction, and uses crucially the interpretation of 
conserved quantities as time translation ``Hamiltonians", from Theorem \ref{timeTranslation} 
and Corollary \ref{conscor}.

The three subsequent sections are devoted to the proofs of Theorem \ref{vanderq}, Theorem \ref{Main},
and Theorem \ref{defectqdet}, all in current/constant term form. The method
of proof is radically different, and uses constant term manipulations. 

\subsection{Quadratic relations in $\U_r$: proof of eq. \eqref{quadratic} }

The quadratic relation \eqref{Msys} and the commutation relations \eqref{Mcom} imply the existence of the 
following infinite sequence of quadratic relations (an exchange algebra)
in $\U_r$. 
\begin{thm} Let $M_n=M_{1,n}$ denote generators in $\U_r$. Given $n\in \Z, 1\leq p$, 
\begin{equation*}
[M_{n},M_{n+p}]_q + [M_{n+p-1},M_{n+1}]_q = 0.
\end{equation*}
\end{thm}

\begin{proof} Let $\phi_{n,\ell} = [M_n,M_{n+\ell}]_q+[M_{n+\ell-1},M_{n+1}]_q$ for $\ell\geq 1.$ 
We prove that $\phi_{n,\ell} =0$ by induction on $\ell$.
First note that $\phi_{n,1}=2[M_n,M_{n+1}]_q=0$ due to the commutation relations \eqref{Mcom},
and $\phi_{n,2}=q(q^{-1}M_{n}M_{n+2}-M_{n+1}^2+q^{-1}M_{2,n+1}) -(q M_{n+2}M_{n}-M_{n+1}^2+M_{2,n+1})=0$ as a linear combination of the quadratic relations \eqref{leftM} and \eqref{Msys} for $\al=1$ and $n\to n+1$. 
Using Eq. \eqref{trans} of Corollary \ref{conscor}, we compute for all $n,\ell$:
$$[C_1,[M_n,M_{n+\ell}]_q] = (1-q)\left([M_{n+1},M_{n+\ell}]_q + [M_n, M_{n+\ell+1}]_q\right),
$$
This implies the inductive step: 
$$
[\frac{C_1}{1-q}, \phi_{n,\ell}]- \phi_{n+1,\ell-1} =\phi_{n,\ell+1}.
$$
Suppose $\phi_{n,\ell'}=0$ for all $n$ and $\ell'\leq \ell$, then the left hand side of this equation is 
zero, hence $\phi_{n,\ell+1}=0$, and the Theorem follows by induction on $\ell$.
\end{proof}

\begin{remark} As mentioned before,
in addition to the quadratic relation \eqref{quadratic}, we have a cubic relation which follows automatically in $\U_r$.
It is instructive to re-derive it directly from eqns. \eqref{Msys} and \eqref{Mcom}. We have the following.
\begin{thm}
Let $n,p,k\in \Z$. Then the generators $M_n$ satisfy
\begin{equation}\label{cubic}
{\rm Sym} \left[ M_n,\left[M_{p-1}, M_{k+1}\right]\right]=0,
\end{equation}
where by Sym we mean the symmetrization over the letters $n,p,k$.
\end{thm}
\begin{proof}
Let $\phi_{n,k,p}$ denote the left hand side of \eqref{cubic}. Let $\delta = \max(n,k,p)-\min(n,k,p)$, and the proof will follow by induction on $\delta$.

If $\delta=0$, then the commutators $[M_n, M_{n-1}M_{n+1}]=0=[M_{n},M_{n+1}M_{n-1}]$ due to the commutation among $M_n$ and $M_{n\pm 1}$. 

The function $\phi_{n,k,p}$ is symmetric with respect to $n,k,p$ so we assume $n\leq k \leq p=n+\delta$. Using
\begin{equation}\label{ind}
(1-q)^{-1} [C_1,\phi_{n,k,p}] = \phi_{n+1,k,p} + \phi_{n,k+1,p} + \phi_{n,k,p+1},
\end{equation}
will provide an inductive step. If $\phi_{n,k,p}=0$ for all $n\leq k \leq p=n+\delta$, then the left hand side is zero. The right hand side will contain one or more terms with $\delta'>\delta$.
\begin{itemize}
\item If $n\leq k<p$, then all terms on the right side have $\delta'\leq \delta$ except for $\phi_{n,k,p+1}$, which therefore vanishes.
\item If $n<k=p$ then two terms on the right hand side are equal to each other by the symmetry of $\phi_{n,k,p}$, and both have $\delta'=\delta+1$. Thus, $\phi_{n,k,k+1}=0$ for $n<k$.
\item If $n=k=p$, the three terms on the right are equal to each other by symmetry, and therefore 
$\phi_{k,k,k+1}=0$ since the left side vanishes.
\end{itemize}

This argument does not show that $\phi_{n,k+1,k+1}=0$. To see this, use Equation \eqref{Crmin} instead:
\begin{equation}\label{din}
(1-q^{-1})^{-1} [A^{-1}C_r,\phi_{n+1,k+1,k+1}]- \phi_{n+1,k,k+1} - \phi_{n+1,k+1,k} = \phi_{n,k+1,k+1}\end{equation}
which allows to conclude that $\phi_{n,k+1,k+1}=0$, as all the terms on the left have width $\delta=k-n$.
The Theorem follows by induction on $\delta$.
\end{proof}
\end{remark}

\subsection{Recursion relation for $M_{\al,n}$: proof of eq. \eqref{qcommInd}}

We consider the subset of generators $M_{n}:=M_{1,n}, n\in\Z$. The next result shows that the elements $(1-q)^{\al-1}M_{\al,m}$, with $\al>1$, are polynomials in these generators with coefficients in $\Z[q]$.


\begin{thm} \label{recurMalpha}
Let $r+2\geq \al>1, n\in \Z$, and $M_{\al,n}$ defined in $\U_r$ via \eqref{Msys} and \eqref{Mcom}.  
We have the relations:
\begin{eqnarray}
(-1)^\al (q-1) M_{\al,n} &=& [M_{\al-1,n-1}, M_{n+\al-1}]_{q^\al} \label{leftmal}\\
&=& [M_{n-\al+1},M_{\al-1,n+1}]_{q^\al} \label{rightmal}
\end{eqnarray}
\end{thm}

\begin{proof}
The proof is by induction on $\al$.
Define 
$$\phi_{\al,n} = (-1)^\al (q-1) M_{\al,n} - [M_{\al-1,n-1},M_{n+\al-1}]_{q^\al}.$$
If $\al=1$ this is zero due to the $q$-commutation relation
$$\phi_{1,n} = (1-q) M_{1,n} - M_{0,n-1}M_{1,n}+q M_{1,n} M_{0,n-1} = 0$$
because $M_{0,m}=1$ for all $m$.

The inductive step is as follows. Assume $\phi_{\al-1,m}=0$ for all $m\in \Z$. 
We will show that the quantity $M_{\al-1,n+1}\phi_{\al,n}+\phi_{\al-1,n+1} M_{\al,n}=0$ for all $n\in \Z$. 
This implies $\phi_{\al,n}=0$. To show this, write:
\begin{eqnarray*}
M_{\al-1,n+1}\phi_{\al,n} + \phi_{\al-1,n+1} M_{\al,n}&=& 
(-1)^\al (q-1)\left(M_{\al-1,n+1}M_{\al,n} -M_{\al-1,n+1}M_{\al,n}\right) \\ &
&\hskip-1in-M_{\al-1,n+1}[M_{\al-1,n-1},M_{n+\al-1}]_{q^\al} - [M_{\al-2,n},M_{n+\al-1}]_{q^{\al-1}}M_{\al,n}.
\end{eqnarray*}
The first line vanishes, and we expand the second line, denoting $m=n+\al-1$:
$$
M_{\al-1,n+1}M_{\al-1,n-1}M_{m} - q^\al M_{\al-1,n+1} M_{m}M_{\al-1,n-1}
+M_{\al-2,n} M_{m} M_{\al,n} - q^{\al-1} M_{m} M_{\al-2,n}M_{\al,n} .
$$
In the second term we use $M_{\al-1,n+1}M_{1,n+\al-1}=q^{\al-2}M_{1,n+\al-1}M_{\al-1,n+1}$. Then we use \eqref{Msys} to replace $M_{\al-1,n+1}M_{\al-1,n-1}$ by $q^{-\al+1} (M_{\al-1,n}^2-M_{\al-2,n}M_{\al,n})$ in the first and second terms. Expanding, and regrouping the six terms, we have
\begin{eqnarray*}
& &q^{-\al+1} M_{\al-1,n}^2M_m - q^{\al-1} M_m M_{\al-1,n}^2 \\
&+& q^{\al-1}M_m M_{\al-2,n}M_{\al,n} - q^{\al-1} M_m M_{\al-2,n} M_{\al,n} \\
&-&q^{-\al+1} M_{\al-2,n}M_{\al,n}M_m + M_{\al-2,n} M_m M_{\al,n}.
\end{eqnarray*}
The first line is $q^{-\al+1}[M_{\al-1,n}^2,M_{1,n+\al-1}]_{q^{2\al-2}}$, which vanishes because of the $q$-commutation \eqref{comM}, which can be used with $\al\to\al-1$, $p=\al-1$ and $\beta=1$. The second line is zero, and the third line is $-q^{-\al+1}M_{\al-2,n}[M_{\al,n},M_{1,n+\al-1}]_{q^{\al-1}}$, which again vanishes using the $q$-commutator \eqref{comM}.
This completes the proof of \eqref{leftmal}.
The second relation \eqref{rightmal} is obtained by acting on \eqref{leftmal} with the ``time-reversal" anti-automorphism $\tau$ of the algebra $\U_r$ defined as follows.
\begin{defn}\label{deftau}
The time-reversal anti-automorphism $\tau$ is defined on $\U_r$ via:
\begin{eqnarray*}
\tau(X Y)&=&\tau(Y)\tau(X) \qquad (X,Y\in \U_r)\\
\tau(M_{\al,n})&=& q^{-\al n} M_{\al,-n}\ ,\quad 
\tau(A)= A^{-1}\ ,\quad 
\tau(\Delta)= \Delta\ , \quad
\tau(q)=q
\end{eqnarray*}
\end{defn}
It is a straightforward to check that $\tau$ preserves the relations (\ref{Msys}-\ref{Mcom}), as well as \eqref{Ur}
and \eqref{degA}. We finally get:
\begin{eqnarray*}\tau\left([M_{\al-1,n-1}, M_{n+\al-1}]_{q^\al}\right)&=&[q^{-n-\al+1}M_{-n-\al+1},q^{(\al-1)(-n+1)}M_{\al-1,-n+1}]_{q^\al}\\
&=&q^{-\al n}[M_{-n-\al+1},M_{\al-1,-n+1}]_{q^\al}=(-1)^\al (q-1) q^{-n\al}M_{\al,-n}
\end{eqnarray*}
which boils down to \eqref{rightmal} upon changing $n\to -n$.
Theorem \ref{recurMalpha} follows.
\end{proof}

%
%

We note that both expressions in Theorem \ref{recurMalpha} are equivalent. By iteration, they lead to the nested commutator expression \eqref{firstMalpha} for $M_{\al,n}$.

\begin{example}\label{exfewms}
We have:
\begin{eqnarray*}
(q-1)M_{2,n}&=&M_{n-1}M_{n+1}-q^2M_{n+1}M_{n-1}\\
-(q-1)^2M_{3,n}&=&M_{n-2}M_nM_{n+2}-q^2M_{n-2}M_{n+2}M_{n}-q^3M_nM_{n+2}M_{n-2}
+q^5M_{n+2}M_{n}M_{n-2}\\
-(q-1)^3M_{4,n}&=&M_{n-3}M_{n-1}M_{n+1}M_{n+3}-q^2M_{n-3}M_{n-1}M_{n+3}M_{n+1}
-q^3M_{n-3}M_{n+1}M_{n+3}M_{n-1}\\
&& +q^5 M_{n-3}M_{n+3}M_{n+1}M_{n-1}-q^4 M_{n-1}M_{n+1}M_{n+3}M_{n-3}
+q^6M_{n-1}M_{n+3}M_{n+1}M_{n-3}\\
&& -q^7M_{n+1}M_{n+3}M_{n-1}M_{n-3}+q^9M_{n+3}M_{n+1}M_{n-1}M_{n-3}
\end{eqnarray*}
\end{example}

The current formulation \eqref{alpharec} of eq. \eqref{rightmal} allows to write an explicit constant term expression
for the current ${\mathfrak m}_\al(z)$  in terms of solely ${\mathfrak m}(z)$ as follows.

\begin{thm}\label{Malct}
We have the expression:
$$(-1)^{\frac{\al(\al-1)}{2}}(1-q)^{\al-1}\, {\mathfrak m}_\al(z)=
CT_{u_1,...,u_\al}\left( P_\al(u_1,...,u_\al)\, \prod_{i=1}^\al
{\mathfrak m}(u_i)\, \delta(u_1\cdots u_\al/z)\right)$$
with $CT$ as in \eqref{CTdef}, and
where the Laurent polynomials $P_\al(u_1,...,u_\al)$ are defined recursively via:
\begin{eqnarray*}
P_1(u_1)&=&1\\
P_{\al+1}(u_1,...,u_{\al+1})&=&\frac{(u_1)^\al}{u_2\cdots u_{\al+1}}\, P_\al(u_2,...,u_{\al+1})-
q^{\al+1}\frac{(u_{\al+1})^\al}{u_1\cdots u_{\al}}P_\al(u_1,...,u_\al)\quad (\al\geq 0)
\end{eqnarray*}
\end{thm}
\begin{proof} By induction, using \eqref{alpharec}.
\end{proof}

\begin{example}
The first few $P$'s read:
\begin{eqnarray*}
P_1(u_1)\!\!\!\!&=&\!\!\!\!1\\
P_2(u_1,u_2)\!\!\!\!&=&\!\!\!\!\frac{u_1}{u_2}-q^2\frac{u_2}{u_1}\\
P_3(u_1,u_2,u_3)\!\!\!\!&=&\!\!\!\!\frac{u_1^2}{u_3^2}-q^2\frac{u_1^2}{u_2^2}
-q^3\frac{u_3^2}{u_2^{2}}+q^5\frac{u_3^2}{u_1^{2}}\\
P_4(u_1,u_2,u_3,u_4)\!\!\!\!&=&\!\!\!\!\frac{u_1^3u_2}{u_3u_4^3}
-q^2\frac{u_1^3u_2}{u_3^3u_4}-q^3\frac{u_1^3u_4}{u_2u_3^3}+q^5\frac{u_1^3u_4}{u_2^3u_3}-q^4\frac{u_4^3u_1}{u_2u_3^3}+q^6\frac{u_4^3u_1}{u_2^3u_3}+q^7\frac{u_4^3u_3}{u_1u_2^3}-q^9\frac{u_4^3u_3}{u_1^3u_2}
\end{eqnarray*}
These lead immediately to the expressions of Example \ref{exfewms}.
\end{example}

\subsection{A quantum determinant expression for $M_{\al,n}$: proof of Theorem \ref{vanderq}}

In this section, we derive an alternative polynomial expression of $M_{\al,n}$ in terms of the $M_n$'s.
The disadvantage of the expression \eqref{firstMalpha} for $M_{\al,n}$ is that it only displays
$M_{\al,n}$ as a polynomial of the $M_n$'s with coefficients rational in $q$. This can be greatly improved by applying 
manipulations to the constant term identity of Theorem \ref{Malct}. These manipulations simply amount to rewritings 
of the constant term identities modulo the exchange relation \eqref{quadratic}, which states in
current form that
$(z-q w){\mathfrak m}(z){\mathfrak m}(w)$ is skew-symmetric under the interchange $z\leftrightarrow w$. 
The main result  is an expression for $M_{\al,n}$ as a polynomial of the $M_n$'s,
with coefficients in $\Z[q]$. The latter is equivalent to \eqref{firstMalpha} modulo the quadratic relations \eqref{quadratic}.

%
Theorem \ref{vanderq} is rephrased as follows.

\begin{thm}\label{vandermal}
For $\al=1,2,...,r+1$ defining:
\begin{equation}
\label{qdet}
{\mathfrak r}_\al(z):=CT_{u_1,\cdots,u_\al}\left( \Delta_q(u_1,...,u_\al) 
\,\prod_{i=1}^\al {\mathfrak m}(u_i)\, \delta(u_1...u_\al/z)\right)\ ,
\end{equation}
with $\Delta_q$ as in \eqref{qvander}, 
then we have
\begin{equation}{\mathfrak m}_\al(z)={\mathfrak r}_\al(z)\qquad (\al=1,2,...,r+1)\end{equation}
\end{thm}

This immediately implies the following.

\begin{cor}
The constant term formula of Theorem \ref{vandermal} translates into an explicit expression for $M_{\al,n}$
in $\Z[q][M_{n-\al+1},M_{n-\al+2},...,M_{n+\al-1}]$. 
\end{cor}
\begin{proof} 
By taking the coefficient of $z^n$ in \eqref{qdet}, and using the explicit expansion of $\Delta_q(u_1,...,u_\al)$
into a $\Z[q]$ sum of Laurent monomials in the variables $u_1,u_2,...,u_\al$, whose constant term in \eqref{qdet} 
yields monomials in the variables $M_{n-\al+1},M_{n-\al+2},...,M_{n+\al-1}$.
\end{proof}

\begin{example}\label{Mex}
Expanding the coefficient of $z^n$ in \eqref{qdet} for the first few values of $\al=2,3$ gives:
\begin{eqnarray}
M_{2,n}&=&M_n^2 -q M_{n+1}M_{n-1}\label{mextwo}\\
M_{3,n}&=&M_n^3-qM_{n+1}M_{n-1}M_n-q(1-q) M_{n+1}M_nM_{n-1}-q M_nM_{n+1}M_{n-1}\label{mexthree}\\
&&\qquad\qquad +q^2 M_{n+2}M_{n-1}^2+q^2M_{n+1}^2M_{n-2}-q^3M_{n+2}M_nM_{n-2}
\nonumber 
\end{eqnarray}
corresponding to the following expansions of the $q$-Vandermonde products:
\begin{eqnarray*}
\Delta_q(u_1,u_2)&=&1-q \frac{u_2}{u_1}\\
\Delta_q(u_1,u_2,u_3)&=&1-q \frac{u_2}{u_1}-q(1-q) \frac{u_3}{u_1}-q \frac{u_3}{u_2}
+q^2 \frac{u_2u_3}{u_1^2}+q^2 \frac{u_3^2}{u_1 u_2}-q^3\frac{u_3^2}{u_1^2}
\end{eqnarray*}
These expressions agree with those of Example \ref{exfewms} modulo the quadratic relations \eqref{quadratic}.
For instance, starting with the first line of Example \ref{exfewms}, we recover \eqref{mextwo} by computing:
\begin{eqnarray*}
(q-1)M_{2,n}&=&(q-1)M_{2,n}-([M_{n-1},M_{n+1}]_q+[M_n,M_n]_q)\\
&=&M_{n-1}M_{n+1}-q^2M_{n+1}M_{n-1}-(M_{n-1}M_{n+1}-qM_{n+1}M_{n-1}-(q-1)M_n^2)\\
&=&(q-1)(M_n^2-q M_{n+1}M_{n-1})
\end{eqnarray*}
where we have used the quadratic relation \eqref{quadratic} for $(n,p)\to (n-1,2)$.
\end{example}

\begin{remark}\label{classidet}
Theorem \ref{vandermal} is the quantum generalization of the discrete Wronskian determinant of the classical case (when $q=1$, $A=1$ and $\Delta=1$ and $Q_{\al,n}=M_{\al,n}$)  for which \cite{DFK10}:
$$M_{\al,n}=\det_{1\leq a,b\leq \al}\left(M_{n+b-a}\right)$$
leading to the classical generating function:
\begin{eqnarray*}
{m}_\al(z)&=&\sum_{n\in \Z} z^n M_{\al,n}=\sum_{n\in \Z}z^n
CT_{u_1,\cdots,u_\al}\left\{ \det_{1\leq a,b\leq \al}\left(\frac{1}{u_a^{n+b-a}}\sum_{m_a\in \Z}(u_a)^{m_a} M_{m_a}\right)
\right\} \\
&=&CT_{u_1,\cdots,u_\al} \left\{ \det_{1\leq a,b\leq \al}\left((u_a)^{a-b} {m}(u_a)\right)
\delta(u_1...u_\al/z)\right\}\\
&=&CT_{u_1,\cdots,u_\al} \left\{ \prod_{1\leq a<b\leq \al} \left(1-\frac{u_b}{u_a}\right)
\prod_{i=1}^\al {m}(u_i)
\delta(u_1...u_\al/z)\right\}
\end{eqnarray*}
where we have used the multi-linearity of the determinant.
\end{remark}


The remainder of this section is devoted to the proof of Theorem \ref{vandermal}.

Our main tool for constant term manipulations is the following.
\begin{lemma}\label{skelem}
For any Laurent polynomial $P(u,v)$ symmetric in $(u,v)$, we have
$$CT_{u,v}\left\{ (u-qv) P(u,v) \, {\mathfrak m}(u){\mathfrak m}(v) \delta(uv/z)\right\}=0 $$
for all $z$.
\end{lemma}
\begin{proof}
By the exchange relation \eqref{current_quadratic}, $(u-qv) P(u,v) \, {\mathfrak m}(u){\mathfrak m}(v)\delta(uv/z)$ is skew-symmetric 
under the interchange $u\leftrightarrow v$.
However the constant term is invariant under the change of variables $(u,v)\mapsto (v,u)$, it must therefore vanish.
\end{proof}

\begin{defn}
For short, we use the notation
$$\langle f(u_1,...,u_\al) \rangle_{\al}(u) := CT_{u_1,...u_\al}\left(\Delta_q(u_1,...,u_\al)\, f(u_1,...,u_\al) 
\, {\mathfrak m}(u_1)\cdots {\mathfrak m}(u_\al)\delta(u_1\cdots u_\al/u)\right) $$
\end{defn}

In the following, we'll make successive applications of Lemma \ref{skelem}. As an example, we have the following.

\begin{lemma}\label{kels}
For any Laurent series $f(x)\in \C((x))$, any symmetric Laurent polynomial $g$ of the variables $u_1,...,u_\al$,
and $1\leq i\leq \al$, we have the following constant term identities.
\begin{eqnarray}
\langle f(u_i) \, g \rangle_{\al}(u) &=& (-1)^{k} \Big\langle \frac{(u_{i-k})^k}{u_iu_{i-1}\cdots u_{i-k+1}} 
f(u_{i-k})\, g\Big\rangle_{\al}(u)
\quad (0\leq k\leq i-1)\label{down}\\
&=& (-1)^{p}\Big\langle  \frac{u_iu_{i+1}\cdots u_{i+p-1}}{(u_{i+p})^p}f(u_{i+p})\, g
\Big\rangle_{\al}(u) \quad (0\leq p\leq \al-i+1) \label{up}
\end{eqnarray}
\end{lemma}
\begin{proof}
To prove \eqref{down},
we apply Lemma \ref{skelem} successively to pairs $(u_{i-1},u_i)$, $(u_{i-2},u_{i-1})$, ..., $(u_{i-k},u_{i-k+1})$.
Note that only finitely many terms in the series for $f$ contribute, so that only a truncated Laurent polynomial part of $f$
contributes.
The first application, with the Laurent polynomial 
$P(u_{i-1},u_i)=g(u_{i-1}^{-1}f(u_{i-1})+u_i^{-1}f(u_i))\Delta_q(u_1,...,u_\al)/(1-q u_i/u_{i-1})$, symmetric in $(u_{i-1},u_i)$,
allows to substitute:
$$ \left(1-q\frac{u_i}{u_{i-1}}\right)f(u_i){\mathfrak m}(u_{i-1}){\mathfrak m}(u_i)\to -\frac{u_{i-1}}{u_i} 
\left(1-q\frac{u_i}{u_{i-1}}\right)f(u_{i-1}){\mathfrak m}(u_{i-1}){\mathfrak m}(u_i)$$
into the constant term,
so that $\langle f(u_i)\rangle_{\al}(u)=-\langle \frac{u_{i-1}}{u_i}f(u_{i-1})\rangle_{\al}(u)$. The second application, with 
the Laurent polynomial $P(u_{i-2},u_{i-1})=g u_i^{-1}(\frac{u_{i-2}}{u_{i-1}}f(u_{i-2})+\frac{u_{i-1}}{u_{i-2}}f(u_{i-1}))\Delta_q(u_1,...,u_\al)/(1-q u_{i-1}/u_{i-2})$, symmetric in $(u_{i-2},u_{i-1})$, allows to substitute:
$$-\frac{u_{i-1}}{u_i}\left(1-q\frac{u_{i-1}}{u_{i-2}}\right)f(u_{i-1}){\mathfrak m}(u_{i-2}){\mathfrak m}(u_{i-1})\to \frac{(-u_{i-2})^2}{u_{i-1}u_i} \left(1-q\frac{u_{i-1}}{u_{i-2}}\right)f(u_{i-2}){\mathfrak m}(u_{i-2}){\mathfrak m}(u_{i-1})$$
Iterating this $k$ times results in \eqref{down}. 
The proof of \eqref{up} is similar, and proceeds by applying successively Lemma \ref{skelem} to the pairs
$(u_i,u_{i+1})$, $(u_{i+1},u_{i+2})$, ..., $(u_{i+k-1},u_{i+k})$.
\end{proof}

Most of the uses of Lemma \ref{skelem} are summed up in the following.

\begin{cor}\label{coruseful}
For any Laurent polynomial $g$ symmetric in the variables $u_1,u_2,...,u_\al$, we have:
\begin{equation}\label{range}
\langle (u_i)^m \, g\rangle_{\al}(u)=0 \qquad {\rm for}\quad m=-i+1,-i+2,...,-i+\al
\end{equation}
while for $m=-i$:
\begin{equation}\label{negbound}
\langle (u_i)^{-i} \, g\rangle_{\al}(u)=\left\langle \frac{(-1)^{i-1}}{u_1u_2\cdots u_i} \, g\right\rangle_\al(u)
\end{equation}
and for $m=\al-i+1$:
\begin{equation}\label{posbound}
\langle (u_i)^{\al-i+1} \, g\rangle_{\al}(u)=\left\langle (-1)^{\al-i}\ u_iu_{i+1}\cdots u_{\al}  \, g\right\rangle_\al(u)
\end{equation}
\end{cor}
\begin{proof}
By direct application of Lemma \ref{kels}, for $f(x)=x^m$. 
For $m\leq 0$, starting from \eqref{down} with $k=-m-1$,
we have:
$$ \langle (u_i)^m g\, \rangle_{\al}(u)=(-1)^{m-1} \left\langle \frac{1}{u_iu_{i-1}\cdots u_{i+m+1}}\, g\right\rangle_{\al}(u) $$
For $m+i=0$ this gives \eqref{negbound}, while for $m+i>0$,
applying Lemma \ref{skelem} to the pair $(u_{i+m},u_{i+m+1})$ gives a zero answer, and \eqref{range} follows
for $m=-i+1,-i+2,...,0$. 
For $m>0$, 
starting from \eqref{up} with $p=m-1$, we have:
$$ \langle (u_i)^m g\, \rangle_{\al}(u)=(-1)^{m-1} \left\langle u_iu_{i+1}\cdots u_{i+m-1}\, g\right\rangle_{\al}(u) $$
For $m+i=\al+1$ this gives \eqref{posbound}, while for $m+i\leq \al$,
applying Lemma \ref{skelem} to the pair $(u_{i+m-1},u_{i+m})$ gives a zero answer, and
\eqref{range} follows
for $m=1,2,...,\al-i$. 
\end{proof}

\begin{lemma}\label{otherqdet}
The defining relation \eqref{qdet} implies the following:
\begin{equation}
\label{qdetprime}
{\mathfrak r}_\al(z)=CT_{u_1,\cdots,u_{\al}}\left( \frac{(-u_1)^{\al-1}}{u_2u_3\cdots u_{\al}}\,
\Delta_q(u_1,u_2,...,u_\al)\, 
\left(\prod_{i=1}^{\al}{\mathfrak m}(u_i)\right) \delta(u_1...u_{\al}/z)\right)
\end{equation}
\end{lemma}
\begin{proof}
Direct application of \eqref{down}, with $f(x)=g=1$, $i=\al$ and $k=\al-1$.
\end{proof}

We will now prove Theorem \ref{vandermal}, by induction on $\al$. 
For short, denote for $1\leq a<b\leq \al$ by:
$$\Delta_q[a,b]:=\Delta_q(u_a,u_{a+1},...,u_b)
=\prod_{a\leq i<j \leq b}\left(1-q \frac{u_j}{u_i}\right) $$
The theorem holds clearly for $\al=1$. Assume the result holds for $\al-1$.
We use the recursion relation \eqref{alpharec} and the recursion hypothesis
to write:
\begin{eqnarray*}
&&\!\!\!\!\!\!\!\!\!\!\!\!\!\!\!(1-q){\mathfrak m}_\al(z)=
CT_{u,v}\left(
\left(\frac{(-u)^{\al-1}}{v}\, {\mathfrak m}(u){\mathfrak m}_{\al-1}(v)
-q^\al \frac{(-v)^{\al-1}}{u}\, {\mathfrak m}_{\al-1}(u){\mathfrak m}(v)\right)\, \delta(uv/z)\right)\\
&=&CT_{u_1,...,u_\al}\left\{\left(\frac{(-u_1)^{\al-1}}{u_2u_3\cdots u_\al}\,
\Delta_q[2,\al]-q\,\frac{(-q u_\al)^{\al-1}}{u_1...u_{\al-1}}
\Delta_q[1,\al-1]\right)
\prod_{a=1}^{\al} {\mathfrak m}(u_a)\delta(u_1 u_2\cdots u_{\al}/z)\right\}
\end{eqnarray*}
On the other hand, let us write $(1-q){\mathfrak r}_\al(z)={\mathfrak r}_\al(z)- q {\mathfrak r}_\al(z)$
by using \eqref{qdetprime} for the first term and \eqref{qdet} for the second, and compute:
\begin{eqnarray*}
&&(1-q)({\mathfrak r}_\al(z)-{\mathfrak m}_\al(z))=CT_{u_1,...,u_\al}\Bigg\{\Bigg(\frac{(-u_1)^{\al-1}}{u_2u_3\cdots u_\al}\, \Delta_q[2,\al]\,
\left(\prod_{a=2}^{\al} \left(1-q \frac{u_a}{u_1}\right)-1\right)\\
&&\qquad -q\, \Delta_q[1,\al-1]\,
\left( \prod_{a=1}^{\al-1} \left(1-q \frac{u_\al}{u_a}\right)-\frac{(-q u_\al)^{\al-1}}{u_1...u_{\al-1}}\right)\Bigg)
\prod_{a=1}^{\al} {\mathfrak m}(u_a)\delta(u_1 u_2\cdots u_{\al}/z)\Bigg\}
\end{eqnarray*}
Let us expand the first factor as:
\begin{eqnarray}\frac{(-u_1)^{\al-1}}{u_2u_3\cdots u_\al} \left(\prod_{a=2}^{\al} \left(1-q \frac{u_a}{u_1}\right)-1\right)
&=&\prod_{i=2}^\al \left( q-\frac{u_1}{u_i}\right)-\frac{(-u_1)^{\al-1}}{u_2u_3\cdots u_\al}\nonumber \\
&=&q\sum_{k=0}^{\al-2}\frac{(-u_1)^k}{u_\al u_{\al-1}\cdots u_{\al-k+1}}\prod_{i=2}^{\al-k-1} \left(q-\frac{u_1}{u_i}\right)
\label{contribs}
\end{eqnarray}
%
and consider the contribution corresponding to some $k>0$ in the summation, as a function
of  $u_{\al+1-k},u_{\al-k}$: it is the product of $\Delta_q[2,\al]=(1-qu_{\al+1-k}/u_{\al-k})\times$ 
a symmetric expression in $(u_{\al+1-k},u_{\al-k})$, by a Laurent polynomial of the form $P_k/u_{\al+1-k}$ where $P_k$ is independent of $u_{\al+1-k}$ and $u_{\al-k}$.
We may apply Lemma \ref{skelem} to $(u,v)=(u_{\al-k},u_{\al+1-k})$ to conclude that
this contribution to the constant term vanishes, as it reads $(u_{\al-k}-qu_{\al-k+1})$ times a symmetric expression
in $(u_{\al-k},u_{\al+1-k})$. We are therefore left with the 
single contribution corresponding to $k=0$:
$$ q\, \Delta_q[2,\al]\, \prod_{a=2}^{\al-1}\left(q-\frac{u_1}{u_a}\right) =q \frac{(-u_1)^{\al-2}}{u_2\cdots u_{\al-1}} \, \frac{\Delta[1,\al]}{1-q\frac{u_\al}{u_1}}$$

Analogously, the second factor in \eqref{contribs} is expanded as:
$$
\left(1-q \frac{u_\al}{u_a}\right) -\frac{(-q u_\al)^{\al-1}}{u_1\cdots u_{\al-1}}
=\sum_{k=0}^{\al-2} 
\frac{(-q u_\al)^{k}}{u_{\al-1}u_{\al-2}\cdots u_{\al-k}}\prod_{a=1}^{\al-2-k} 
\left(1-q \frac{u_\al}{u_a}\right)
$$
Again, the contribution corresponding to some $k>0$ in the summation, viewed as a function
of $(u_{\al-k-1},u_{\al-k})$ is the product of  $\Delta_q[1,\al-1]$ by a Laurent polynomial of the form
$Q_k/u_{\al-k}$, where $Q_k$ is independent of $u_{\al-k-1}$ and $u_{\al-k}$. 
Again, this contribution to the constant term vanishes by applying Lemma \ref{skelem}
to $(u,v)=(u_{\al-k-1},u_{\al-k})$, and we are left with the $k=0$ contribution:
$$q\, \Delta_q[1,\al-1] \, \prod_{a=1}^{\al-2}\left(1-q \frac{u_\al}{u_a}\right)=q\, \frac{\Delta_q[1,\al]}{1-q\frac{u_\al}{u_{\al-1}}}$$
We arrive at:
$$(1-q)({\mathfrak r}_\al(z)-{\mathfrak m}_\al(z))=q\, 
\Big\langle \frac{(-u_1)^{\al-2}}{u_2\cdots u_{\al-1}} \, \frac{1}{1-q\frac{u_\al}{u_1}}
-\frac{1}{1-q\frac{u_\al}{u_{\al-1}}}\Big\rangle_\al(z)
$$
We finally apply the result of Lemma \ref{kels}, \eqref{down} for $f(x)=\frac{1}{1-q \frac{u_{\al}}{x}}$, $g=1$
and $i=\al-1$, $k=\al-2$, to compute:
$$\Big\langle \frac{1}{1-q\frac{u_\al}{u_{\al-1}}}-\frac{(-u_1)^{\al-2}}{u_2\cdots u_{\al-1}}\, \frac{1}{1-q\frac{u_\al}{u_{1}} }
\Big\rangle_{\al}(z)=0$$
We conclude that ${\mathfrak r}_\al(z)-{\mathfrak m}_\al(z)=0$, thus completing the proof of Theorem \ref{vandermal}.

\subsection{Proof of Theorem \ref{Main}}

This section is devoted to the proof of Theorem \ref{Main}. We have already shown that the generators of $\U_r$
obey the quadratic relation \eqref{quadratic}. Conversely, let us assume we have a set of generators
$M_n$ obeying the quadratic relations \eqref{quadratic}. Moreover, let us {\it define} the quantities $M_{\al,n}$
via the quantum determinant expression of Theorem \ref{vanderq}. As mentioned before, this definition
is equivalent to \eqref{firstMalpha}, but is much more convenient for current manipulations and proofs.
Theorem \ref{Main} follows from the following:

\begin{thm}\label{goback}
We consider generators $M_n$ obeying the quadratic relations \eqref{quadratic}, and their associated 
current ${\mathfrak m}(z)$. For $\al\geq 0$, we furthermore {\it define} a family of degree $\al$
polynomials $M_{\al,n}\in \Z[q][M_{n-\al+1},...,M_{n+\al-1}]$ via the formula of Theorem \ref{vanderq}, namely:
$$M_{\al,n}:=CT_{u_1,\cdots,u_\al}\left( \Delta_q(u_1,...,u_\al) 
\,\prod_{i=1}^\al {\mathfrak m}(u_i)\, (u_1...u_\al)^{-n}\right)$$
with $M_{0,n}:=1$.
Then the following statements hold true, independently of $r$:

\noindent{ (i) }\ The polynomials $M_{\al,n}$ obey the renormalized quantum Q-system relations \eqref{Msys},
namely we have the following identity:
\begin{eqnarray}
&&\qquad CT_{u_1,...,u_{2\al}}\Big\{ \Big( \Delta_q[1,\al+1]\Delta_q[\al+2,2\al]\label{tobeproved}\\
&&-\left. \left.\left(1-q^\al \frac{u_{\al+1}u_{\al+2}\cdots u_{2\al}}{u_1 u_2 \cdots u_\al}\right) 
\Delta_q[1,\al]\Delta_q[\al+1,2\al]\right)\prod_{i=1}^{2\al} {\mathfrak m}(u_i) \delta(u_1\cdots u_{2\al}/z)  \right\}=0 
\nonumber 
\end{eqnarray}

\noindent{ (ii) }\ The polynomials $M_{\al,n}$ obey the quantum commutation relations \eqref{Mcom},
namely for all $\al,\beta\geq 0$ and $\epsilon=0,1$:
\begin{eqnarray}
&&CT_{u_1,...,u_{\al+\beta}}\Big\{ \Big( \frac{\Delta_q[1,\al]\Delta_q[\al+1,\al+\beta]}{(u_{\al+1}u_{\al+2}\cdots u_{\al+\beta})^\epsilon}\label{tobeproved2}\\
&&\qquad\qquad -\left. \left. q^{{\rm Min}(\al,\beta)\epsilon} \frac{\Delta_q[1,\beta]\Delta_q[\beta+1,\al+\beta]}{(u_{1}u_{2}\cdots u_{\beta})^\epsilon}\right)\prod_{i=1}^{\al+\beta} {\mathfrak m}(u_i) \delta(u_1\cdots u_{2\al}/z)  \right\}=0 
\nonumber 
\end{eqnarray}

\end{thm}

\subsubsection{Proof of statement (i)}

The proof of statement (i) \eqref{tobeproved} goes by rewriting the contribution of $\Delta_q[1,\al+1]\Delta_q[\al+1,2\al]$ in two different ways,
both using only Lemma \ref{skelem}.

Let us first rewrite
\begin{eqnarray*}
&&\Delta_q[1,\al+1]\Delta_q[\al+1,2\al]=\Delta_q[1,\al+1]\Delta_q[\al+2,2\al]\prod_{j=\al+2}^{2\al} \left(1-q\frac{u_j}{u_{\al+1}}\right)\\
&&\qquad =\Delta_q[1,\al+1]\Delta_q[\al+2,2\al]\, \sum_{k=0}^{\al-1} (-q)^k\sum_{\al+2\leq j_1<\cdots<j_k\leq 2\al}
\frac{u_{j_1}\cdots u_{j_k}}{(u_{\al+1})^k}
\end{eqnarray*}
All the contributions with $k>0$ are readily seen to vanish, by use of Corollary \ref{coruseful} for $\al\to \al+1$,
$m=-k$ and $i=\al+1$. We are left with the contribution of $k=0$, namely 
\begin{equation}\label{finalone} \Delta_q[1,\al+1]\Delta_q[\al+2,2\al] \end{equation}


Similarly, we write:
\begin{eqnarray*}
&&\Delta_q[1,\al+1]\Delta_q[\al+1,2\al]=\Delta_q[1,\al]\Delta_q[\al+1,2\al]\prod_{i=1}^{\al} 
\left(1-q\frac{u_{\al+1}}{u_i}\right)\\
&&\qquad =\Delta_q[1,\al]\Delta_q[\al+1,2\al] \sum_{k=0}^{\al} \sum_{1\leq i_1<\cdots <i_k\leq \al}
\frac{(-qu_{\al+1})^k}{u_{i_1}u_{i_2}\cdots u_{i_k}}
\end{eqnarray*}
Again, using Corollary \ref{coruseful} for $m=k$, $i=1$ (and renaming variables $u_i\to u_{i+\al}$), 
we see that all the terms with $0<k<\al$ 
in the sum contribute zero to the constant term, and we are left with the two contributions $k=0$ and $k=\al$:
$$\Delta_q[1,\al]\Delta_q[\al+1,2\al] \left\{ 1+\frac{(-qu_{\al+1})^\al}{u_{\al}u_{\al-1}\cdots u_{1}}\right\}$$
The second term is rewritten using \eqref{posbound} for $i=1$, together with a renaming of variables $u_i\to u_{i+\al}$ ,
to finally get:
\begin{equation}\label{finaltwo} \Delta_q[1,\al]\Delta_q[\al+1,2\al] \left\{  
1-q^\al \frac{u_{\al+1}u_{\al+2}\cdots u_{2\al}}{u_1u_2\cdots u_\al}\right\}
\end{equation} 
We conclude that both \eqref{finalone} and \eqref{finaltwo} have identical contributions to the constant term,
which amounts to \eqref{tobeproved}. This completes the proof of statement (i) of Theorem \ref{goback}.

\subsubsection{Proof of statement (ii)}

We first prove the statement (ii) \eqref{tobeproved2} for $\epsilon=0$. 
In this case, we may assume $0<\al<\beta$ without loss of generality.
Let us compute in two different ways the contribution of the l.c.m. of the two products
$\Delta_q[1,\al]\Delta_q[\al+1,\al+\beta]$ and $\Delta_q[1,\beta]\Delta_q[\beta+1,\al+\beta]$, to the constant term.
We define
$$K(z):=CT_{u_1,...,u_{\al+\beta}}\left(\frac{\Delta_q[1,\al+\beta]}{\prod_{i\in [1,\al]\atop  j\in [\beta+1,\beta+\al]}
\left(1-q\frac{u_j}{u_i}\right)} \prod_{i=1}^{\al+\beta}{\mathfrak m}(u_i)
\delta(u_1\cdots u_{\al+\beta}/z) \right)$$
The l.c.m. can be written in the two following ways:
\begin{eqnarray}
\frac{\Delta_q[1,\al+\beta]}{\prod_{i\in [1,\al]\atop  j\in [\beta+1,\beta+\al]}
\left(1-q\frac{u_j}{u_i}\right)}&=& \Delta_q[1,\al]\Delta_q[\al+1,\al+\beta] 
\prod_{i\in [1,\al]\atop  j\in [\al+1,\beta]}\left(1-q\frac{u_j}{u_i}\right) \label{lcmone}\\
&=&\Delta_q[1,\beta]\Delta_q[\beta+1,\al+\beta]  
\prod_{i\in [\al+1,\beta]\atop  j\in [\beta+1,\beta+\al]}\left(1-q\frac{u_j}{u_i}\right)\label{lcmtwo}
\end{eqnarray}

Let us expand both products on the r.h.s. as sums of Laurent monomials, respectively:
\begin{eqnarray}
\prod_{i\in [1,\al]\atop  j\in [\al+1,\beta]}\left(1-q\frac{u_j}{u_i}\right)&=&
\sum_{k=0}^{\al(\beta-\al)}
\sum_{1\leq i_1\leq\cdots i_k\leq \al \atop \al+1\leq j_1\leq\cdots \leq j_k\leq \beta} (-q)^k \frac{u_{j_1}\cdots u_{j_k}}{u_{i_1}\cdots u_{i_k}}\label{smone}\\
\prod_{i\in [\al+1,\beta]\atop  j\in [\beta+1,\beta+\al]}\left(1-q\frac{u_j}{u_i}\right)&=&
\sum_{k=0}^{\al(\beta-\al)}
\sum_{\al+1\leq i_1\leq\cdots i_k\leq \beta \atop \beta+1\leq j_1\leq\cdots \leq j_k\leq \beta+\al} 
(-q)^k \frac{u_{j_1}\cdots u_{j_k}}{u_{i_1}\cdots u_{i_k}}\label{smtwo}
\end{eqnarray}
Let us pick an arbitrary Laurent monomial in the first sum \eqref{smone}, with $k>0$.
Let $j:=j_k$ be the largest index appearing in the numerator, then the monomial may be written
as $(u_{j})^a \mu$ where the monomial $\mu$ depends only on variables with indices $i_\ell,j_\ell<j$,
while $0<a\leq\al$, $j\leq \beta$. Applying Corollary \ref{coruseful} for $m=a$, and noting that
$j+a\leq \al+\beta$, we find that only the term $k=0$
of the sum, with value $1$, contributes to the constant term.


Similarly, pick an arbitrary Laurent monomial in the second sum \eqref{smtwo}, with $k>0$.
Let $i:=i_1$ be the smallest index appearing in the denominator, then the monomial may be written as 
$u_{i}^{-a}\lambda$, where the monomial $\lambda$ depends only on variables with indices $i_\ell,j_\ell>i$,
while $0<a\leq \al$, $i\geq \al+1$. Applying Corollary \ref{coruseful} for $m=-a$, and noting that
$i-a\geq 1$, we find that only the term $k=0$
of the sum, with value $1$, contributes to the constant term.


We conclude that
\begin{eqnarray*}&&CT_{u_1,...,u_{\al+\beta}}\Big(\left\{ \Delta_q[1,\al]\Delta_q[\al+1,\al+\beta] 
-\Delta_q[1,\beta]\Delta_q[\beta+1,\al+\beta]  \right\} \\
&&\qquad\qquad\qquad\qquad 
\times \prod_{i=1}^{\al+\beta}{\mathfrak m}(u_i)
\delta(u_1\cdots u_{\al+\beta}/z) \Big)=K(z)-K(z)=0 
\end{eqnarray*}
which amounts to the statement (ii) of Theorem \ref{goback} for $\epsilon=0$.

We finally turn to the proof of the statement (ii) \eqref{tobeproved2} for $\epsilon=1$. 
We must distinguish three cases: $\al=\beta$, $\al<\beta$ and $\al>\beta$.

Imitating the proof for (i), let us compute
$$L(z):=CT_{u_1,...,u_{2\al}}\left\{
\frac{\Delta_q[1,\al+1]\Delta_q[\al+1,2\al]}{u_{\al+1}u_{\al+2}\cdots u_{2\al}}
\prod_{i=1}^{2\al}{\mathfrak m}(u_i)
\delta(u_1\cdots u_{2\al}/z) \right\}
$$
in two different ways. First, we write:
\begin{eqnarray*}\frac{\Delta_q[1,\al+1]\Delta_q[\al+1,2\al]}{u_{\al+1}u_{\al+2}\cdots u_{2\al}}&=&
\frac{\Delta_q[1,\al]\Delta_q[\al+1,2\al]}{u_{\al+1}u_{\al+2}\cdots u_{2\al}}\prod_{i=1}^\al\left(1-q\frac{u_{\al+1}}{u_i}\right)\\
&=&\frac{\Delta_q[1,\al]\Delta_q[\al+1,2\al]}{u_{\al+1}u_{\al+2}\cdots u_{2\al}}\sum_{k=0}^\al
\sum_{1\leq i_1<\cdots <i_k\leq\al} \frac{(-qu_{\al+1})^k}{u_{i_1}\cdots u_{i_k}}
\end{eqnarray*}
By Corollary \ref{coruseful} for $m=k$ we find that only the terms $k=0$ and $k=\al$ contribute:
$$\frac{\Delta_q[1,\al]\Delta_q[\al+1,2\al]}{u_{\al+1}u_{\al+2}\cdots u_{2\al}}
\left( 1+\frac{(-qu_{\al+1})^\al}{u_{1}u_2\cdots u_{\al}}\right)$$
The second term is rewritten using \eqref{posbound} for $i=1$, together with a renaming of variables $u_i\to u_{i+\al}$ ,
to finally get:
$$\frac{\Delta_q[1,\al]\Delta_q[\al+1,2\al]}{u_{\al+1}u_{\al+2}\cdots u_{2\al}}\left( 1-q^\al\frac{u_{\al+1}u_{\al+2}\cdots u_{2\al}}{u_{1}u_2\cdots u_{\al}}\right)$$
Next we write:
\begin{eqnarray*}\frac{\Delta_q[1,\al+1]\Delta_q[\al+1,2\al]}{u_{\al+1}u_{\al+2}\cdots u_{2\al}}&=&
\frac{\Delta_q[1,\al+1]\Delta_q[\al+2,2\al]}{u_{\al+1}u_{\al+2}\cdots u_{2\al}}
\prod_{j=\al+2}^{2\al}\left(1-q\frac{u_j}{u_{\al+2}}\right)\\
&=&\frac{\Delta_q[1,\al+1]\Delta_q[\al+2,2\al]}{u_{\al+1}u_{\al+2}\cdots u_{2\al}}\sum_{k=0}^{\al-1}(-q)^k
\sum_{\al+1\leq j_1<\cdots <j_k\leq 2\al} \frac{u_{j_1}\cdots u_{j_k}}{(u_{\al+1})^k}
\end{eqnarray*}
By Corollary \ref{coruseful} with $\al\to \al+1$,
$m=-k-1$, $i=\al+1$, $m+i=\al-k\geq 1$, we see that all the terms in the sum have a vanishing contribution, 
hence $L(z)=0$.
We conclude that
$$CT_{u_1,...,u_{2\al}}\left\{ \Delta_q[1,\al]\Delta_q[\al+1,2\al]
\left( \frac{1}{u_{\al+1}u_{\al+2}\cdots u_{2\al}}-\frac{q^\al}{u_{1}u_2\cdots u_{\al}}\right)
\prod_{i=1}^{2\al}{\mathfrak m}(u_i)
\delta(u_1\cdots u_{2\al}/z) \right\}=0$$
which is nothing but \eqref{tobeproved2} for $\al=\beta$ and $\epsilon=1$.

We now turn to the case $\al<\beta$, and $\epsilon=1$.
Imitating the proof for $\epsilon=0$, let us introduce 
$$M(z):=CT_{u_1,...,u_{\al+\beta}}\Big(\frac{\Delta_q[1,\al+\beta]}{\prod_{i\in [1,\al]\atop  j\in [\beta+1,\beta+\al]}
\left(1-q\frac{u_j}{u_i}\right)}\frac{\prod_{i=1}^{\al+\beta}{\mathfrak m}(u_i)
\delta(u_1\cdots u_{\al+\beta}/z)}{u_{\al+1}u_{\al+2}\cdots u_{\al+\beta}} \Big)
$$
We now write $M(z)$ in two ways, using (\ref{lcmone}-\ref{lcmtwo}), and the associated expansions (\ref{smone}-\ref{smtwo}),
now with an extra prefactor of $1/(u_{\al+1}u_{\al+2}\cdots u_{\al+\beta})$. The prefactor does not affect the argument for the first sum \eqref{smone}, whose only contribution is still $k=0$, hence: 
\begin{equation}\label{oneM}
M(z)=CT_{u_1,...,u_{\al+\beta}}\left( \frac{\Delta_q[1,\al]\Delta_q[\al+1,\al+\beta]}{u_{\al+1}u_{\al+2}\cdots u_{\al+\beta}}
\prod_{i=1}^{\al+\beta}{\mathfrak m}(u_i)
\delta(u_1\cdots u_{\al+\beta}/z)\right)
\end{equation}
To each term of the second sum \eqref{smtwo}, we apply Corollary \ref{coruseful} for $\al\to\beta$, 
now with $m=-a-1$, $0\leq a\leq \al$ and $i=\al+1$, as the smallest index in the denominator is always $\al+1$.
As $m+i=\al-a\geq 0$, we are left with only the contributions with $a=\al$, of terms of the form 
$(-q)^k\frac{u_{\beta+1}u_{\beta+2}\cdots u_{\beta+\al}}{(u_{\al+1})^{\al}}\frac{\lambda'}{u_{\al+1}u_{\al+2}\cdots u_{\al+\beta}}$, $\lambda'$ a monomial involving variables $u_\ell$
with $\beta\geq \ell>\al+1$. By \eqref{negbound}, with $i=\al$ and $\al\to\al+1$, this contributes the same as:
$$(-1)^\al(-q)^k\frac{u_{\beta+1}u_{\beta+2}\cdots u_{\beta+\al}}{u_1u_2\cdots u_{\al}}\frac{\lambda'}{u_{\al+1}u_{\al+2}\cdots u_{\al+\beta}}=(-1)^{\al}(-q)^k\frac{\lambda'}{u_1u_2\cdots u_\beta}$$
If $\lambda'$ had a non-trivial denominator, i.e. $\lambda'=u_{i'}^{-a'}\lambda''$, 
$i'\geq \al+2$ the smallest label appearing,
we would get a zero contribution from Corollary
\ref{coruseful} with $\al\to \beta$, $m=-a'\geq -\al$, $i'\geq \al+2$, hence $i'+m\geq 2$, and $g=1/(u_1u_2\cdots u_{\beta})$. 
We are left with only $\lambda'=1$, which implies
$k=\al$ and:
\begin{equation}\label{twoM}
M(z)=q^\al CT_{u_1,...,u_{\al+\beta}}\left(\frac{\Delta_q[1,\beta]\Delta_q[\beta+1,\al+\beta]}{u_{1}u_{2}\cdots u_{\beta}}
\prod_{i=1}^{\al+\beta}{\mathfrak m}(u_i)
\delta(u_1\cdots u_{\al+\beta}/z)\right)
\end{equation}
Identifying the expressions \eqref{oneM} and \eqref{twoM} yields \eqref{tobeproved2} for $\al<\beta$ and $\epsilon=1$.

The case $\al>\beta$ is dealt with by interchanging the roles of $\al$ and $\beta$ in the previous result, which amounts to
\eqref{tobeproved2} for $\al>\beta$ and $\epsilon=-1$, which implies the identity for $\epsilon=1$ upon dividing by 
$z=u_1u_2\cdots u_{\al+\beta}$ within the constant term.

The theorem follows.

%
%

\subsection{Conserved quantities as quantum determinants: proof of Theorem \ref{defectqdet}}
\label{secdefect}
This section is devoted to the proof of Theorem \ref{defectqdet}. It turns out to be a consequence of the 
more general formula:


\begin{thm}\label{defecthm}
For all $m=0,1,...,r+1$, we have:
\begin{equation}\label{qdetdefect}
C_m\, {\mathfrak m}_{r+1}(u)=CT_{u_1,...,u_{r+1}}\left(
\frac{\Delta_q(u_1,...,u_{r+1})}{u_1u_{2}\cdots u_{m}} 
\prod_{a=1}^{r+1}{\mathfrak m}(u_a)\delta(u_1\cdots u_{r+1}/u)\right)
\end{equation}
\end{thm}
\begin{proof}
We start from the fact that ${\mathfrak m}_{r+2}(z)=0$ (from \eqref{Ur}), and write:
\begin{eqnarray*}
0&=&CT_{u_1,...,u_{r+2}}\left(\Delta_q[1,r+2]
\prod_{a=1}^{r+2}{\mathfrak m}(u_a)\delta(u_1\cdots u_{r+2}/z)\right)\\
&=&CT_{u_1,...,u_{r+2}}\left(\Delta_q[1,{r+1}]\prod_{a=1}^{r+1}\left( 1-q\frac{u_{r+2}}{u_a}\right)
\prod_{a=1}^{r+2}{\mathfrak m}(u_a)\delta(u_1\cdots u_{r+2}/z)\right)\\
&=&CT_{u_1,...,u_{r+2}}\left(\Delta_q[1,{r+1}]
\sum_{m=0}^{r+1} \frac{(-qu_{r+2})^m}{u_{1}u_2\cdots u_{m}}\prod_{a=m+2}^{r+1}\left( 1-q\frac{u_{r+2}}{u_a}\right)
\prod_{a=1}^{r+2}{\mathfrak m}(u_a)\delta(u_1\cdots u_{r+2}/z)\right)\\
\end{eqnarray*}
where in the last term, the products for $a=r+2,r+3$ are taken to be $1$.
Let us expand the $m$-th contribution for each $m\geq 0$ to the sum as:
$$\frac{(-qu_{r+2})^m}{u_{1}u_2\cdots u_{m}}\prod_{a=m+2}^{r+1}\left( 1-q\frac{u_{r+2}}{u_a}\right)=\frac{(-qu_{r+2})^m}{u_{1}u_2\cdots u_{m}}\left\{ 1-q \sum_{\ell=m+2}^{r+1}\frac{u_{r+2}}{u_{\ell}} \prod_{k=\ell+1}^{r+1}\left( 1-q\frac{u_{r+2}}{u_k}\right)\right\}$$
and apply Lemma \ref{skelem} to each contribution with $\ell\geq m+2$, with 
$(u,v)=(u_{\ell-1},u_\ell)$. As before the $\ell$ term is proportional to $1/u_\ell$ and independent of $u_{\ell-1}$,
hence when multiplied by $\Delta_q[1,r+1]{\mathfrak m}(u_1)\cdots {\mathfrak m}(u_{r+2})$
it is skew-symmetric in $(u_{\ell-1},u_\ell)$, and therefore the constant term vanishes.
The only remaining contribution is $1$, and we conclude that
\begin{equation}\label{lutalfin}
0= CT_{u_1,...,u_{r+2}}\left(\Delta_q[1,{r+1}] \sum_{m=0}^{r+1} \frac{(-qu_{r+2})^m}{u_{1}u_2\cdots u_{m}} \prod_{a=1}^{r+2}{\mathfrak m}(u_a)\delta(u_1\cdots u_{r+2}/z)\right)
\end{equation}
Define
$$c_m(u)=\sum_{n\in \Z}u^n c_{m,n}:=CT_{u_1,...,u_{r+1}}\left(\Delta_q[1,{r+1}] \frac{q^m}{u_{1}u_2\cdots u_{m}} \prod_{a=1}^{r+1}{\mathfrak m}(u_a)\delta(u_1\cdots u_{r+1}/u)\right)$$
then \eqref{lutalfin} is equivalent to the recursion relation:
$$CT_{u_{r+2}}\left(\sum_{m=0}^{r+1}c_m(u)(-u_{r+2})^m{\mathfrak m}(u_{r+2})\delta(uu_{r+2}/z)\right)=0\ \Rightarrow\ 
\sum_{m=0}^{r+1} (-1)^m c_{m,n} M_{n-m}=0 \qquad (n\in \Z)$$
Comparing to \eqref{consar}, we find that $c_{m,n}$ must be proportional to $C_m$,
up to an overall factor depending on $n$ only, namely $c_{m,n}=R_n\, C_m$. $R_n$
is fixed by the value for $m=0$, where $C_0=1$ and $c_0(u)={\mathfrak m}_{r+1}(u)$
by definition. We conclude that $R_{n}=M_{r+1,n}$ and 
$c_m(u)={\mathfrak m}_{r+1}(u)\, C_m$. Finally, by Lemma \ref{commutaC}, 
$C_m$ commutes with $C_{r+1}=A$, and $\Delta\, C_m=q^{m}C_m\Delta$ (by \eqref{comCdelt}),
so that we have $M_{r+1,n}\, C_m=q^{m}C_m\, M_{r+1,n}$,
hence $c_m(u)=q^{m}  C_m\, {\mathfrak m}_{r+1}(u)$, and the theorem follows.
\end{proof}

\begin{example}
In the case $A_1$, $r=1$ of Example \ref{exAone}, picking the coefficient of $u^n$ in \eqref{qdetdefect}, we get:
$$C_1=(M_{n+1}M_n-q M_{n+2}M_{n-1})\, \Delta^{-1}\, A^{-n}$$
Comparing with the constant term (coefficient of $u^0$), we get the conservation law:
$$ M_{n+1}M_n-q M_{n+2}M_{n-1}=q^{-2n}(M_{1}M_0-q M_{2}M_{-1})A^n $$
\end{example}

The explicit formula of Theorem \ref{defectqdet} for $C_m$ is obtained by taking the constant term in $u$
in \eqref{qdetdefect},
which selects the coefficient $C_m\,M_{r+1,0}=C_m\,\Delta$.

Summing the result of Theorem \ref{qdetdefect}
and reinstating the vanishing terms, we may also express the generating polynomial
for the conserved quantities. Noting further that $\Delta\, C(z)=C(q z)\,\Delta$ while $A$ commutes with $C(z)$,
so that $M_{r+1,n}\, C(z)=C(qz)\, M_{r+1,n}$ for all $n\in \Z$, we arrive at the following compact formula.

\begin{cor}\label{currcons}
We have:
\begin{eqnarray*}{\mathfrak m}_{r+1}(u)\, C(z)&=&C(qz)\, {\mathfrak m}_{r+1}(u)\\
&=&CT_{u_1,...,u_{r+1}}\left(\Delta_q(u_1,...,u_{r+1},z)
\prod_{a=1}^{r+1}{\mathfrak m}(u_a)\, \delta(u_1\cdots u_{r+1}/u)\right)
\end{eqnarray*}
\end{cor}

\section{The quantum affine algebra}

In this section we define other currents in $\U_r$, which satisfy relations similar to that of the quantum affine algebra of $\sl_2$ with zero central extension and a non-standard Cartan current valuation, depending on the integer $r$.

\subsection{Definitions of the generating functions $\mathfrak f$ and $\psi^{\pm}$}

So far, we have shown that $\U_r$ is isomorphic to a quotient
of quantum enveloping algebra ${\mathcal U}_{\sqrt{q}}({\mathfrak n}_+[u,u^{-1}])$. Define
the renormalized generating function
\begin{equation}\label{edef}{\mathfrak e}(z):={\mathfrak m}(q^{1/2}z)=\sum_{n\in \Z}z^n q^{n/2} M_n.
\end{equation} 
The generating function ${\mathfrak f}(z)$ is defined using two involutions of $\U_r$.  Recall the time-reversal anti-automorphism $\tau$
from Definition \ref{deftau}. By definition, $\tau({\mathfrak e}(z))={\mathfrak e}(z^{-1})$. 

There is another involutive automorphism of $\U_r$:
\begin{defn}\label{sigmadef}
The Weyl reflection automorphism $\sigma$ is defined on $\U_r$ via:
\begin{eqnarray*}
\sigma(M_{\al,n})&= &A^{-n} \, M_{r+1-\al,n}\, \Delta^{-1},\quad 
\sigma(A)=A^{-1},\quad 
\sigma(\Delta)=\Delta^{-1},\quad 
\sigma(q)=q .
\end{eqnarray*}
\end{defn}
It is straightforward to check that $\sigma$ preserves the relations \eqref{Msys}, \eqref{Mcom} and \eqref{Ur} and that it is an involution.

We define the new generating function ${\mathfrak f}(z)$ as
\begin{equation}\label{fdef}{\mathfrak f}(z):=\sigma \tau ({\mathfrak e}(z))=\sum_{n\in \Z} z^{-n} q^{n/2} A^{-n}M_{r,n}\Delta^{-1}.
\end{equation}


Let $C(z)$ be the generating polynomial of conserved quantities \eqref{genercons}.
We define the two series, $\psi^\pm\in \U_r[[z^{\pm 1}]]$ (Cartan currents), as the expansions of the same rational fraction in either $z$ or $z^{-1}$:
\begin{eqnarray}
{\mathfrak \psi}^+(z)&=&(-q^{-1/2} z)^{r+1}AC(q^{1/2}z)^{-1}C(q^{-1/2}z)^{-1}\label{pplus}\\
&=&(-q^{-1/2} z)^{r+1}A(1+(q^{1/2}+q^{-1/2})C_1 z+ O(z^2))\nonumber\\
{\mathfrak \psi}^-(z)&=&(-q^{1/2}z)^{-r-1}A\left(\frac{C(q^{1/2}z)}{(q^{1/2}z)^{r+1}}\right)^{-1}
\left(\frac{C(q^{-1/2}z)}{(q^{-1/2}z)^{r+1}}\right)^{-1},
\label{pmoins}\\
&=&(-q^{1/2}z)^{-r-1} A^{-1}\left(\hat{C}(q^{-1/2}z^{-1})\hat{C}(q^{1/2} z^{-1})\right)^{-1}\nonumber \\
&=&(-q^{1/2}z)^{-r-1}A^{-1}(1+(q^{1/2}+q^{-1/2})A^{-1}C_r z^{-1}+ O(z^{-2})).\nonumber
\end{eqnarray}
Here, we have introduced the notation
\begin{equation}\label{hatC}
\hat{C}(z):=(-z)^{r+1}A^{-1}C(z^{-1})=\sum_{m=0}^{r+1}(-z)^j A^{-1}C_{r+1-j}=1-zA^{-1}C_r+O(z^2).
\end{equation}

\begin{lemma}\label{Cfixed}
The conserved quantities $C_m$ are invariant under $\sigma \tau$, for all $m=0,1,...,r+1$. Equivalently:
$$\sigma\tau(C(z))=C(z) . $$
\end{lemma}
\begin{proof}
Computing explicitly the action of $\sigma$ and $\tau$ on the Miura operator
$\mu=\mu_n=\mu_{-n-1}$ of \eqref{miura}. 
We have
$\sigma(x_{\al,n})=x_{r+2-\al,n}^{-1}$ and
$\tau(x_{\al,n})=x_{\al,-n-1}^{-1}$. As $\sigma$ and $\tau$ both commute with the time shift operator $D$,
we have:
\begin{eqnarray*}
\sigma(\mu_n)&=&\prod_{\al=r+1}^{1}(D -x_{r+2-\al,n}^{-1})=\prod_{\al=1}^{r+1}(D -x_{\al,n}^{-1}),\\
\tau(\mu_{-n-1})&=& \prod_{\al=1}^{r+1} (D-x_{\al,n}^{-1}).
\end{eqnarray*}
This implies that $\sigma(C_m)=\tau(C_m)$ for all $m$. Moreover, $\sigma$ and $\tau$ are involutions by definition, hence $\sigma\tau(C_m)=C_m$ for all $m$.
\end{proof}

\begin{cor}
\begin{equation}\label{invsigtau}
\sigma\tau(\psi^{\pm}(z))=\psi^{\pm}(z)
\end{equation}
\end{cor}
\begin{proof}
The coefficients of both series $\psi^\pm$ are polynomials of the conserved quantities $C_m$, hence the Corollary follows from Lemma \ref{Cfixed}. 
\end{proof}

\subsection{Drinfeld-type relations}

We now come to the main theorem of this section, in the form of relations between the currents 
${\mathfrak e}(z),{\mathfrak f}(z)$ and the series ${\mathfrak \psi}^\pm(z)$.

\begin{thm}\label{pretoro}
The currents ${\mathfrak e}(z),{\mathfrak f}(z)$ and series ${\mathfrak \psi}^\pm(z)$ obey the following 
relations:
\begin{eqnarray}
(z-q w){\mathfrak e}(z)\, {\mathfrak e}(w)+(w-q z) {\mathfrak e}(w)\,  {\mathfrak e}(z)&=&0\label{mmcom}\\
(z-q^{-1} w){\mathfrak f}(z)\, {\mathfrak f}(w)+(w-q^{-1} z) {\mathfrak f}(w)\,  {\mathfrak f}(z)&=&0\label{nmcom}\\
(z-q w){\mathfrak \psi}^\pm(z)\, {\mathfrak e}(w)+(w-q z) {\mathfrak e}(w)\,  {\mathfrak \psi}^\pm(z)&=&0\label{pmcom}\\
(z-q^{-1} w){\mathfrak \psi}^\pm(z)\, {\mathfrak f}(w)+(w-q^{-1} z) {\mathfrak f}(w)\,  {\mathfrak \psi}^\pm(z)&=&0\label{pncom}\\
{[} {\mathfrak e}(z), {\mathfrak f}(w){]}=(1-q)\delta(z/w)\left({\mathfrak \psi}^+(z)-{\mathfrak \psi}^-(z)\right)\!\!\!\!\!\!\!\!\!\!\!\!&&
\label{comn}
\end{eqnarray}
\end{thm}
\begin{proof}
Eq.\eqref{mmcom} is equivalent \eqref{current_quadratic}.
Eq.\eqref{nmcom} is obtained by applying $\sigma\tau$ to \eqref{MMcom}: as $\tau$ is an anti-automorphism, 
the order of the currents is switched. 

To prove \eqref{pmcom}, we use the formulation  \eqref{Cmcom} of 
Theorem \ref{timeTranslation}:
\begin{equation}\label{iter} C(z)\, {\mathfrak e}(w)\frac{q^{1/2} z-w}{q^{-1/2}z-w}= {\mathfrak e}(w)\, C(z) 
\end{equation}
with $C(z)$ as in \eqref{genercons}.
using the commutation $A^{-1}M_n=q M_n A^{-1}$,
$$A^{-1}C(q^{-1/2}z)C(q^{1/2}z) \,{\mathfrak e}(w)\frac{z-w}{q^{-1}z-w}\frac{q z-w}{z-w}
=q\, {\mathfrak e}(w)\, A^{-1}C(q^{-1/2}z)C(q^{1/2}z)$$
and \eqref{pmcom} follows for $\psi^+$. The equation for $\psi^-$ follows, as both $\psi_{\pm}(z)$
are the same rational fraction of $z$.
Eq. \eqref{pncom} follows from \eqref{pmcom} by applying $\sigma\tau$, and using \eqref{invsigtau}.

It remains to prove the commutation relation  \eqref{comn}. Define $\rho_{n,p}$ from the equation
\begin{equation}\label{tobedone}
{[} {\mathfrak e}(z), {\mathfrak f}(w){]}=\sum_{n,p\in \Z} z^{n+p} w^{-n} q^{n+p/2} {[} M_{n+p}, A^{-n}M_{r,n}\Delta^{-1}{]}=:
\sum_{n,p\in \Z} z^{n+p} w^{-n} \rho_{n,p}.
\end{equation}
Then $\rho_{n,p}$ can be computed directly for sufficiently small values of $|p|$:
\begin{equation}\label{gaprho}
\rho_{n,p}=q^{n+p/2} {[} M_{n+p}, A^{-n}M_{r,n}\Delta^{-1}{]}=q^{p/2}A^{-n}{[} M_{n+p},{M}_{r,n}{]}_{q^{-p}}\Delta^{-1}=0 
\quad {\rm for}\ p=0,\pm 1,...,\pm r
\end{equation}
by the relations \eqref{comM} for $\al=1,\beta=r$. Furthermore, when $p=\pm(r+1)$, 
\begin{eqnarray*}
\rho_{n,r+1}&=&q^{\frac{r+1}{2}}A^{-n}{[} M_{n+r+1},{M}_{r,n}{]}_{q^{-r-1}}\Delta^{-1}= (-1)^{r}(q-1)q^{-\frac{r+1}{2}}A^{-n}
M_{r+1,n+1}\Delta^{-1}\\
&=&(-1)^{r}(q-1)q^{-\frac{r+1}{2}}A\\
\rho_{n,-r-1}&=&q^{-\frac{r+1}{2}}A^{-n}{[} M_{n-r-1},{M}_{r,n}{]}_{q^{r+1}}\Delta^{-1}=(-1)^{r+1}(q-1)q^{-\frac{r+1}{2}}A^{-n} 
M_{r+1,n-1} \Delta^{-1}\\
&=&(-1)^{r+1}(q-1)q^{-\frac{r+1}{2}}A^{-1}
\end{eqnarray*}
by \eqref{leftmal} and \eqref{rightmal} for $\al=r+1$. 

Using the linear recursion \eqref{consar}, define $\sigma_{n,p}$ from the equation
\begin{equation}\label{rhosig}
 \sum_{j=0}^{r+1}(-q^{1/2})^j C_j \rho_{n,p-j} =-q^{n+p/2} \sum_{j=0}^{r+1}(-1)^j C_j A^{-n}M_{r,n}\Delta^{-1} M_{n+p-j}=:\sigma_{n,p} \quad
(n,p\in \Z)
\end{equation}
with the initial values $\sigma_{n,p}=0$ for $p=1,2,...,r$ (by \eqref{gaprho}), and
$$\sigma_{n,0}=(-1)^{r+1}C_{r+1}\rho_{n,-r-1}=(q-1)q^{-\frac{r+1}{2}}, \qquad 
\sigma_{n,r+1}=C_0\rho_{n,r+1}=(-1)^{r}(q-1)q^{-\frac{r+1}{2}}A.$$
Moreover, \eqref{rrec} implies the right recursion relation:
\begin{equation}\label{recusigma}
\sum_{j=0}^{r+1}\sigma_{n,p+j} (-q^{-1/2})^{j}C_{r+1-j}=0.
\end{equation}
Then $\sigma_{n,p}$ is entirely fixed by the initial conditions. Indeed, we may consider the recursion relation \eqref{recusigma}
as a descending recursion on $p$ for $\sigma_{n,p}$ for $p\leq r$, with initial conditions $0,0,...,0,(q-1)q^{-\frac{r+1}{2}}$
for $p=r,r-1,...,1,0$, respectively, which fixes entirely all $\sigma_{n,p}$ for $p\leq 0$. On the other hand, we may view it as an ascending
recursion relation on $p$ for $\sigma_{n,p}$ for $p\geq 1$, with initial conditions $0,0,...,0,(-1)^{r}(q-1)q^{-\frac{r+1}{2}}A$
for $p=1,2,...,r,r+1$, which fixes entirely all $\sigma_{n,p}$ for $p\geq r+1$.
The result may be rewritten as: 
$$\sigma_{n,+}(z):=\sum_{p\geq 1} z^p\sigma_{n,p}= (-1)^{r}(q-1)(q^{-1/2}z)^{r+1}A C(q^{-1/2}z)^{-1}$$
as an identity between power series of $z$
while
$$\sigma_{n,-}(z^{-1}):=\sum_{p\leq 0} z^p\sigma_{n,p}=(-1)^{r+1}(q-1)q^{-\frac{r+1}{2}} 
A\left(\frac{C(q^{-1/2}z)}{(q^{-1/2}z)^{r+1}}\right)^{-1}
=(q-1)q^{-\frac{r+1}{2}}\hat{C}(q^{1/2} z^{-1})^{-1}$$
with $\hat{C}$ as in \eqref{hatC},
as an identity between power series of $z^{-1}$. Introducing similarly the series
$\rho_{n,\pm}(z^{\pm 1}):=\sum_{\pm p\geq 0}z^p \rho_{n,p}$, we may rephrase \eqref{rhosig} as:
$$ C(q^{1/2}z)\rho_{n,\pm}(z^{\pm 1})=\sigma_{n,\pm}(z^{\pm 1}) $$
resulting in
\begin{eqnarray*}
\rho_{n,+}(z)&=&(-1)^{r}(q-1)(q^{-1/2}z)^{r+1}A C(q^{1/2}z)^{-1}C(q^{-1/2}z)^{-1}\\
&=&(1-q){\mathfrak \psi}^+(z)\\
\rho_{n,-}(z^{-1})&=&(-1)^{r+1}(q-1)(q^{1/2}z)^{-r-1} A\left(\frac{C(q^{1/2}z)}{(q^{1/2}z)^{r+1}}\frac{C(q^{-1/2}z)}{(q^{-1/2}z)^{r+1}}\right)^{-1}\\
&=&
(-1)^{r+1}(q-1)(q^{1/2}z)^{-r-1} A^{-1}\left(\hat{C}(q^{-1/2} z^{-1})\hat{C}(q^{1/2} z^{-1})\right)^{-1}\\
&=&-(1-q){\mathfrak \psi}^-(z)
\end{eqnarray*}
This shows that $\rho_{n,p}$ is independent of $n$, and we may finally express \eqref{tobedone}
as:
$${[} {\mathfrak e}(z), {\mathfrak f}(w){]}=\sum_{n\in \Z} \left(\frac{z}{w}\right)^n\sum_{p\in \Z}z^p\rho_{n,p}
= \delta(z/w) (\rho_{n,+}(z)+\rho_{n,-}(z))=(1-q)\delta(z/w)({\mathfrak \psi}^+(z)-{\mathfrak \psi}^-(z))$$

This completes the proof of the theorem.
\end{proof}

\section{Discussion/Conclusion}

In this paper, we have expressed the algebra $\U_r$ associated with the $A_r$ quantum $Q$-system as an $r$-dependent quotient
of the quantum enveloping algebra ${\mathcal U}_{\sqrt{q}}({\mathfrak n}_+[u,u^{-1}])$. This was proved using the 
integrable structure underlying the quantum $Q$-system, making use of quantum conserved quantities. 

We have also obtained a remarkably compact expression \eqref{qdetone} for the solutions $M_{\al,n}$ of the 
$M$-system and commutation
relations (\ref{Msys}-\ref{Mcom}) as polynomials of the variables $\{M_i\}_{i\in \Z}$ with coefficients in $\Z[q]$. 
Up to transforming back $M_{\al,n}\to Q_{\al,n}$, and a harmless renormalization,
the latter had been conjectured in \cite{DFKnoncom} to be the result of telescopic products
of increasing principal quasi-minors of the ``discrete Wronskian" matrix 
$W_{\al,n}=\left(M_{n-a+b}\right)_{1\leq a,b \leq \al}$, using Gelfand and Retakh's definition of quasi-determinants for
matrices with non-commuting entries \cite{GR}. This led us to coining the result as a ``quantum determinant".
We note that
the expression \eqref{qdetone} involves a $q$-deformation $\Delta_q(u_1,...,u_\al)$
of the Vandermonde determinant which for $q=1$ reads simply 
$\Delta_1(u_1,...,u_\al)=\det_{1\leq i,j\leq \al}\left( u_i^{i-j}\right)$.
For generic $q$ however $\Delta_q(u_1,...,u_\al)$ is not such a simple determinant, but it turns out to be the
{\it lambda-determinant} of the same Vandermonde matrix $V_\bu=\left( u_i^{i-j}\right)$, 
as defined by Robbins and Rumsey \cite{RR}, for the particular value $\lambda=-q$. 
The latter is known to be related to the classical $T$-system with coefficients, a higher-dimensional version 
of the $Q$-system, itself having a cluster algebra formulation \cite{DFlambda}.
In particular, it was shown that the lambda-determinant of any matrix is the partition function for suitable families of 
non-intersecting  lattice paths whose local weights involve the matrix entries, or alternatively of weighted domino tilings 
of the so-called Aztec diamond. The expression \eqref{qdetone} should therefore allow to interpret the quantum
determinant of the matrix $W_{\al,n}$ as the partition function for suitable families of 
non-intersecting  lattice paths with {\it non-commuting} local weights involving the entries $\{M_i\}_{i\in \Z}$, thus providing 
a non-trivial example of a quantum Gessel-Viennot theorem \cite{GV}.

We note also that expressions like \eqref{qdetone} and more generally the constant term identities of Section \ref{proofsec}
 are closely related to identities in shuffle algebras \cite{NEGshuf}. 

Constructing moreover other currents in order to complete the quantum affine algebra 
${\mathcal U}_{\sqrt{q}}({\mathfrak sl}_2[u,u^{-1}])$, we found that the Cartan currents $\psi^\pm$ have a non-standard
structure, as series respectively of $z,z^{-1}$, both with valuation $r+1$. It is interesting to note that the limit $r\to \infty$
of the $A_r$ quantum $Q$-system makes sense as arising from a cluster algebra of infinite rank. The valuation property
just mentioned implies that $\psi^\pm\to 0$ as $r\to \infty$, so that the currents $\mathfrak e$ and $\mathfrak f$ provide
two commuting copies of ${\mathcal U}_{\sqrt{q}}({\mathfrak n}_+[u,u^{-1}])$, while the quotient by $M_{r+2,n}=0$
is removed.

This non-standard structure of the Cartan currents can be understood in the light of a natural $t$-deformation of the
whole picture. In \cite{DFK15}, we have constructed a representation of the $A_r$  $M$-system 
and commutations (\ref{Msys}-\ref{Mcom}) by means of difference
operators acting on a space of symmetric functions. The latter generalize the so-called dual $q$-Whittaker limit 
$t\to \infty$ of the Macdonald operators, of which Macdonald polynomials are common eigenfunctions \cite{macdo}.
They act on symmetric functions of the variables $x_1,x_2,...,x_{r+1}$
as:
\begin{equation}\label{genq}
{\mathcal M}_{\al,n}=\sum_{I\subset [1,r+1]\atop |I|=\al} \Big(\prod_{i\in I} x_i\Big)^n
\prod_{i\in I\atop j\not \in I} \frac{x_i}{x_i-x_j} \, \prod_{i\in I} D_i \qquad \al\in [0,r+1], n\in \Z ,
\end{equation}
where the ``shift" operator $D_i$ acts on functions of $x_1,x_2,...,x_{r+1}$ by the substitution $x_i\mapsto q x_i$.
It is easy to check that these operators obey the relations of Theorem \ref{pretoro}, by taking:
$$C(z)=\prod_{i=1}^{r+1}(1-z x_i), \quad {\hat C}(z)=  \prod_{i=1}^{r+1}(1-z x_i^{-1}), \quad A=\prod_{i=1}^{r+1} x_i 
\, , \quad \Delta=\prod_{i=1}^{r+1} D_i \, , $$
and noting that:
\begin{eqnarray*}
{\mathfrak e}(z)&=&\sum_{i=1}^{r+1}\delta(q^{1/2}z x_i)\prod_{j\neq i}
\frac{x_i}{x_i-x_j} \, D_i\\
{\mathfrak f}(z)&=&\sum_{i=1}^{r+1}\delta(q^{-1/2}z x_i)\prod_{j\neq i}
\frac{x_j}{x_j-x_i} \, D_i^{-1} .
\end{eqnarray*}

In the sequel to the present paper \cite{DFK16}, we show that there is a natural $t$-deformation
of these generalized Macdonald operators,
within the contex of Double Affine Hecke Algebras, and show that they give rise to a representation of the quantum 
toroidal algebra for ${\mathfrak gl}_1$ of \cite{FJqtoroidal} at level 0. These act on symmetric functions of the variables 
$x_1,x_2,...,x_{r+1}$ as the following difference operators:
\begin{equation}\label{genqt}
{\mathcal M}_{\al,n}^{q,t}=\sum_{I\subset [1,r+1]\atop |I|=\al} \Big(\prod_{i\in I} x_i\Big)^n
\prod_{i\in I\atop j\not \in I} \frac{t x_i -x_j}{x_i-x_j} \, \prod_{i\in I} D_i \qquad \al\in [0,r+1], n\in \Z ,\end{equation}
with $D_i$ as above. These operators generalize the $A_r$ Macdonald operators \cite{macdo}, to which they reduce for $n=0$.
The non-standard structure of the Cartan currents $\psi^\pm$ is simply a consequence of the $t\to\infty$ limit of the
quantum  toroidal algebra for ${\mathfrak gl}_1$.

%
%



\bibliographystyle{alpha}

\bibliography{refs}

\def\cprime{$'$}
\begin{thebibliography}{FJMM12}

\bibitem[BZ05]{BernZel}
Arkady Berenstein and Andrei Zelevinsky.
\newblock Quantum cluster algebras.
\newblock {\em Adv. Math.}, 195(2):405--455, 2005.

\bibitem[DF13]{DFlambda}
Philippe Di~Francesco.
\newblock An inhomogeneous lambda-determinant.
\newblock {\em Electron. J. Combin.}, 20(3):Paper 19, 34, 2013.

\bibitem[DFK10]{DFK10}
Philippe Di~Francesco and Rinat Kedem.
\newblock {$Q$}-systems, heaps, paths and cluster positivity.
\newblock {\em Comm. Math. Phys.}, 293(3):727--802, 2010.

\bibitem[DFK11]{DFKnoncom}
Philippe Di~Francesco and Rinat Kedem.
\newblock Non-commutative integrability, paths and quasi-determinants.
\newblock {\em Adv. Math.}, 228(1):97--152, 2011.

\bibitem[DFK14]{qKR}
Philippe Di~Francesco and Rinat Kedem.
\newblock Quantum cluster algebras and fusion products.
\newblock {\em Int. Math. Res. Not. IMRN}, (10):2593--2642, 2014.

\bibitem[DFK15]{DFK15}
Philippe Di~Francesco and Rinat Kedem.
\newblock Difference equations for graded characters from quantum cluster
  algebra.
\newblock {\em arXiv:1505.01657 [math.RT]}, 2015.

\bibitem[DFK16]{DFK16}
Philippe Di~Francesco and Rinat Kedem.
\newblock Quantum q systems, daha and quantum toroidal algebras.
\newblock {\em work in progress}, 2016.

\bibitem[Dri87]{Drinfeld}
V.~G. Drinfel{\cprime}d.
\newblock Quantum groups.
\newblock In {\em Proceedings of the {I}nternational {C}ongress of
  {M}athematicians, {V}ol. 1, 2 ({B}erkeley, {C}alif., 1986)}, pages 798--820.
  Amer. Math. Soc., Providence, RI, 1987.

\bibitem[FJMM12]{FJqtoroidal}
B.~Feigin, M.~Jimbo, T.~Miwa, and E.~Mukhin.
\newblock Quantum toroidal {$\germ{gl}_1$}-algebra: plane partitions.
\newblock {\em Kyoto J. Math.}, 52(3):621--659, 2012.

\bibitem[FL99]{FL}
Boris Feigin and Sergey Loktev.
\newblock On generalized {K}ostka polynomials and the quantum {V}erlinde rule.
\newblock In {\em Differential topology, infinite-dimensional {L}ie algebras,
  and applications}, volume 194 of {\em Amer. Math. Soc. Transl. Ser. 2}, pages
  61--79. Amer. Math. Soc., Providence, RI, 1999.

\bibitem[FR96]{FR95}
Edward Frenkel and Nikolai Reshetikhin.
\newblock Quantum affine algebras and deformations of the {V}irasoro and
  {W}-algebras.
\newblock {\em Comm. Math. Phys.}, 178(1):237--264, 1996.

\bibitem[GR97]{GR}
I.~Gelfand and V.~Retakh.
\newblock Quasideterminants. {I}.
\newblock {\em Selecta Math. (N.S.)}, 3(4):517--546, 1997.

\bibitem[GV85]{GV}
Ira Gessel and G{\'e}rard Viennot.
\newblock Binomial determinants, paths, and hook length formulae.
\newblock {\em Adv. in Math.}, 58(3):300--321, 1985.

\bibitem[Kap97]{spherical_hall}
M.~M. Kapranov.
\newblock Eisenstein series and quantum affine algebras.
\newblock {\em J. Math. Sci. (New York)}, 84(5):1311--1360, 1997.
\newblock Algebraic geometry, 7.

\bibitem[Ked08]{Ke07}
Rinat Kedem.
\newblock {$Q$}-systems as cluster algebras.
\newblock {\em J. Phys. A}, 41(19):194011, 14, 2008.

\bibitem[KN99]{kinoum}
A.~N. Kirillov and M.~Noumi.
\newblock {$q$}-difference raising operators for {M}acdonald polynomials and
  the integrality of transition coefficients.
\newblock In {\em Algebraic methods and {$q$}-special functions ({M}ontr\'eal,
  {QC}, 1996)}, volume~22 of {\em CRM Proc. Lecture Notes}, pages 227--243.
  Amer. Math. Soc., Providence, RI, 1999.

\bibitem[Mac95]{macdo}
I.~G. Macdonald.
\newblock {\em Symmetric functions and {H}all polynomials}.
\newblock Oxford Mathematical Monographs. The Clarendon Press, Oxford
  University Press, New York, second edition, 1995.
\newblock With contributions by A. Zelevinsky, Oxford Science Publications.

\bibitem[Neg14]{NEGshuf}
Andrei Negut.
\newblock The shuffle algebra revisited.
\newblock {\em Int. Math. Res. Not. IMRN}, (22):6242--6275, 2014.

\bibitem[RR86]{RR}
David~P. Robbins and Howard Rumsey, Jr.
\newblock Determinants and alternating sign matrices.
\newblock {\em Adv. in Math.}, 62(2):169--184, 1986.

\end{thebibliography}

\end{document}